\documentclass[11pt,a4paper]{amsart}

\usepackage[utf8]{inputenc}
\usepackage[english]{babel}

\usepackage{hyperref}

\usepackage[foot]{amsaddr}

\usepackage{amsmath,amsfonts,amssymb}
\usepackage{lmodern}
\usepackage[left=2.8cm,right=2.8cm,top=2.8cm,bottom=2.8cm]{geometry}

\usepackage{mathtools}

\usepackage{bm}			
\DeclareMathAlphabet{\mathbbold}{U}{bbold}{m}{n}	

\usepackage{cancel}

\usepackage[usenames,dvipsnames]{xcolor}
\hypersetup{
    colorlinks,
    linkcolor={red!50!black},
    citecolor={blue!50!black},
    urlcolor={blue!80!black}
}

\newtheorem{theorem}{Theorem}
\newtheorem{lemma}[theorem]{Lemma}
\newtheorem{proposition}[theorem]{Proposition}
\newtheorem{corollary}[theorem]{Corollary}
\theoremstyle{definition}
\newtheorem{definition}[theorem]{Definition}
\theoremstyle{remark}
\newtheorem{rmk}{Remark}

\newtheorem{example}{Example}
\newtheorem*{example*}{Example}

\newcommand{\distr}{\mathcal{D}}
\newcommand{\g}{\gamma}

\newcommand{\V}{\mathbf{V}}

\newcommand{\R}{\mathbb{R}}
\renewcommand{\H}{\mathbb{H}}
\newcommand{\eps}{\varepsilon}
\newcommand{\J}{\mathcal{J}}
\newcommand{\lam}{\lambda}
\newcommand{\ver}{\mathcal{V}}
\newcommand{\wt}{\widetilde}
\newcommand{\Rcan}{\mathfrak{R}}
\newcommand{\Riccan}{\mathfrak{Ric}}
\newcommand{\Ric}{\mathrm{Ric}}

\newcommand{\f}{\mathsf{f}} 
\renewcommand{\d}{\mathsf{d}} 

\DeclareMathOperator{\spn}{\mathrm{span}}
\DeclareMathOperator{\trace}{\mathrm{Tr}}
\DeclareMathOperator{\Sec}{\mathrm{Sec}}
\DeclareMathOperator{\rank}{\mathrm{rank}}
\DeclareMathOperator{\diam}{\mathrm{diam}}
\DeclareMathOperator{\dive}{\mathrm{div}}
\DeclareMathOperator{\sinc}{\mathrm{sinc}}

\DeclareMathOperator{\grad}{\mathrm{grad}}

\allowdisplaybreaks

\title[Sub-Riemannian curvature and diameter bounds for 3-Sasakian manifolds]{Sub-Riemannian Ricci curvatures and universal diameter bounds for 3-Sasakian manifolds}

\date{\today}

\author[Luca Rizzi]{Luca Rizzi$^{1,2}$}
\address{$^1$Univ. Grenoble Alpes, CNRS, Institut Fourier, F-38000 Grenoble, France}
\address{$^2$Inria, team GECO \& CMAP, \'Ecole Polytechnique, CNRS, Universit\'e Paris-Saclay, Palaiseau, France (past institution)}
\email{\href{mailto:luca.rizzi@univ-grenoble-alpes.fr}{luca.rizzi@univ-grenoble-alpes.fr}}

\author[Pavel Silveira]{Pavel Silveira$^3$}
\address{$^3$Leibniz Universit\"at Hannover, Institut f\"ur analysis}
\email{\href{mailto:psilveir@math.uni-hannover.de}{psilveir@math.uni-hannover.de}}


\begin{document}

\begin{abstract}
For a fat sub-Riemannian structure, we introduce three canonical Ricci curvatures in the sense of Agrachev-Zelenko-Li. Under appropriate bounds we prove comparison theorems for conjugate lengths, Bonnet-Myers type results and Laplacian comparison theorems for the intrinsic sub-Laplacian.

As an application, we consider the sub-Riemannian structure of $3$-Sasakian manifolds, for which we provide explicit curvature formulas. We prove that any complete $3$-Sasakian structure of dimension $4d+3$, with $d>1$, has sub-Riemannian diameter bounded by $\pi$. When $d=1$, a similar statement holds under additional Ricci bounds. These results are sharp for the natural sub-Riemannian structure on $\mathbb{S}^{4d+3}$ of the quaternionic Hopf fibrations:
\begin{equation*}
\mathbb{S}^3  \hookrightarrow \mathbb{S}^{4d+3} \to \mathbb{HP}^d,
\end{equation*}
whose exact sub-Riemannian diameter is $\pi$, for all $d \geq 1$.
\end{abstract}

\maketitle

\section{Introduction and results}\label{s:intro}

\subsection{Sub-Riemannian geometry}
A sub-Rieman\-nian manifold is a triple $(M,\distr,g)$, where $M$ is a smooth, connected manifold of dimension $n \geq 3$, $\distr$ is a vector distribution of constant rank $k \leq n$ and $g$ is a smooth metric on $\distr$. The distribution is bracket-generating, that is 
\begin{equation}
\spn\{[X_{i_1},[X_{i_2},[\ldots,[X_{i_{m-1}},X_{i_m}]]]]\mid m \geq 1\}_q = T_q M, \qquad \forall q \in M,
\end{equation}
for some (and then any) set $X_1,\ldots,X_k \in \Gamma(\distr)$ of local generators for $\distr$. 
If $\rank(\distr) =k$ and $\dim M = n$, we say that $(M,\distr,g)$ is a sub-Rieman\-nian structure of type $(k,n)$.

A \emph{horizontal curve} $\gamma : [0,T] \to \R$ is a Lipschitz continuous path such that $\dot\gamma(t) \in \distr_{\gamma(t)}$ for almost any $t$. Horizontal curves have a well defined length
\begin{equation}
\ell(\gamma) = \int_0^T \sqrt{g(\dot\gamma(t),\dot\gamma(t))}dt.
\end{equation}
The \emph{sub-Rieman\-nian distance} is defined by:
\begin{equation}
\d(x,y) = \inf\{\ell(\gamma)\mid \gamma(0) = x,\, \gamma(T) = y,\, \gamma \text{ horizontal} \}.
\end{equation}
By the Chow-Rashevskii theorem, under the bracket-generating condition, $\d$ is finite and continuous. A sub-Rieman\-nian manifold is complete if $(M,\d)$ is complete as a metric space. Sub-Rieman\-nian geometries include the Riemannian one, when $\distr = TM$.

In this paper we focus on \emph{fat} structures, namely we assume that for any non zero section $X$ of $\distr$, $TM$ is (locally) generated by $\distr$ and $[X,\distr]$. The fat condition is open in the $C^1$ topology, however it gives some restriction on the rank $k$ of the distribution (for example $n \leq 2k -1$, \cite[Prop. 5.6.3]{montgomerybook}). This class includes many popular sub-Rieman\-nian structures, such as contact and quaternionic contact structures. 

\begin{example*}
The main example that motivated our study is the quaternionic Hopf fibration
\begin{equation}
\mathbb{S}^3 \hookrightarrow \mathbb{S}^{4d+3} \xrightarrow{\pi} \mathbb{HP}^d, \qquad d \geq 1.
\end{equation}
Here $\distr = (\ker \pi_*)^\perp$ is the orthogonal complement of the kernel of the differential of the Hopf map $\pi$, and the sub-Rieman\-nian metric is the restriction to $\distr$ of the round one of $\mathbb{S}^{4d+3}$. This is a fat structure of type $(4d,4d+3)$. This example is one of the simplest (non-Carnot) sub-Rieman\-nian structures of corank greater than $1$, and is included in the more general class of $3$-Sasakian structures that we study in Section~\ref{s:3-Sas}.

\end{example*}

\emph{Sub-Rieman\-nian geodesics} are horizontal curves that locally minimize the length between their endpoints. Define the \emph{sub-Rieman\-nian Hamiltonian} $H : T^*M \to \R$ as
\begin{equation}
H(\lambda) := \frac{1}{2}\sum_{i=1}^k \langle \lambda, X_i \rangle^2, \qquad \lambda \in T^*M,
\end{equation}
where $X_1,\ldots,X_k$ is any local orthonormal frame for $\distr$ and $\langle \lambda, \cdot \rangle $ denotes the action of covectors on vectors. Let $\sigma$ be the canonical symplectic $2$-form on $T^*M$. The \emph{Hamiltonian vector field} $\vec{H}$ is defined by $\sigma(\cdot, \vec{H}) = dH$. Then the Hamilton equations are
\begin{equation}\label{eq:Hamiltoneqs}
\dot{\lambda}(t) = \vec{H}(\lambda(t)).
\end{equation}
Solutions of~\eqref{eq:Hamiltoneqs} are called \emph{extremals}, and their projections $\gamma(t) := \pi(\lambda(t))$ on $M$ are geodesics. In the fat setting any (non-trivial) geodesic can be recovered uniquely in this way. This, and many statements that follow, are not true in full generality, as the so-called \emph{abnormal geodesics} can appear. These are minimizing trajectories that might not follow the Hamiltonian dynamic of~\eqref{eq:Hamiltoneqs}, and are related with some challenging open problems in sub-Riemannian geometry \cite{AAA-openproblems}.

The \emph{sub-Rieman\-nian exponential map} $\exp_{q} : T^*_q M \to M$, with base $q \in M$ is
\begin{equation}
\exp_q (\lambda) := \pi \circ e^{\vec{H}}(\lambda), \qquad \lambda \in T_q^*M,
\end{equation}
where $e^{t \vec{H}}(\lambda) : T^*M \to T^*M$ denotes the flow of $\vec{H}$.\footnote{If $(M,\d)$ is complete, $\vec{H}$ is complete, then the domain of $\exp_q$ is the whole $T_q^*M$.} A geodesic then is determined by its initial covector, and its speed $\|\dot\gamma(t)\|= 2H(\lambda)$ is constant. The set of \emph{unit covectors} is
\begin{equation}
U^*M  =\{\lambda  \in T^*M \mid  H(\lambda) = 1/2 \},
\end{equation}
a fiber bundle with fiber $U_q^*M = \mathbb{S}^{k-1}\times \R^{n-k}$. To exclude trivial geodesics, we use the symbol $T^*M \setminus H^{-1}(0)$ to denote the set of covectors with $H(\lambda) \neq 0$.

For $\lambda \in U_q^*M$, the curve $\gamma_\lambda(t)=\exp_q(t \lambda)$ is a \emph{length-parametrized} geodesic with length $\ell(\gamma|_{[0,T]}) = T$. We say that $t_*$ is a \emph{conjugate time} along $\gamma_\lambda$ if $t\lambda$ is a critical point of $\exp_q$. In this case, $\gamma(t_*)$ is a \emph{conjugate point}. The first conjugate time is separated from zero, and geodesics cease to be minimizing after the first conjugate time. Conjugate points are also intertwined with the analytic properties of the underlying structure, for example they affect the behavior of the heat kernel (see \cite{srneel,bbcn} and references therein).


We gave here only the essential ingredients for our analysis; for more details see \cite{nostrolibro,montgomerybook,rifford2014sub}.

\subsection{Canonical Ricci curvatures}

For any fat sub-Rieman\-nian manifold (we assume from now on $k < n$) we introduce three \emph{canonical Ricci curvatures} (see Section~\ref{s:jac})
\begin{equation}
\Riccan^\alpha : U^*M \to \R, \qquad \alpha =a,b,c.
\end{equation}
For any initial covector $\lambda$, the canonical Ricci curvatures computed along the extremal are $\Riccan^\alpha(\lambda(t))$. This is the sub-Rieman\-nian generalization of the classical Ricci curvature tensor $\mathrm{Ric}(\dot\gamma(t))$ evaluated ``along the geodesic'', where the tangent vector $\dot\gamma(t)$ is replaced by its cotangent counterpart $\lambda(t)$. The main theorems we prove are:
\begin{itemize}
\item Bounds for conjugate points along a sub-Rieman\-nian geodesic (Theorems~\ref{t:conjI-intro}, \ref{t:conjII-intro});
\item Bonnet-Myers type results for the sub-Rieman\-nian diameter (Theorems~\ref{t:bmI-intro}, \ref{t:bmII-intro});
\item Laplacian comparison theorems for the canonical sub-Laplacian (Theorem \ref{t:lapl-intro});
\item Formulas for the sub-Riemannian curvature of $3$-Sasakian manifolds (Theorem \ref{t:ricci-3-sas-intro});
\item Sharp bounds for the sub-Riemannian diameter of $3$-Sasakian manifolds (Corollary \ref{c:3sasBM-1}, Proposition~\ref{p:3sasBM-2}) and conjugate distance along a geodesic (Corollary \ref{c:3sas-conj}).
\end{itemize}

\subsection{Two relevant functions} 

We introduce two model functions related with the geodesic flow and their blow-up times. Here $\sqrt{\cdot}$ is the principal value of the square root and, for values of the parameters where a denominator is zero, the functions is understood in the limit. 
\subsubsection{The ``Riemannian'' function} The first function we need is $s_{\kappa_c}: I \to \R$, given by
\begin{equation}\label{eq:Riemfunc}
s_{\kappa_c}(t) = \sqrt{\kappa_c} \cot(\sqrt{\kappa_c} t) = \begin{cases}
\sqrt{\kappa_c} \cot(\sqrt{\kappa_c} t) & \kappa_c >0, \\
\frac{1}{t}  & \kappa_c = 0, \\
\sqrt{|\kappa_c|} \coth(\sqrt{|\kappa_c|} t) & \kappa_c <0. \end{cases}
\end{equation}
The function~\eqref{eq:Riemfunc} is defined on a maximal connected interval $I=(0,\bar{t}(\kappa_c))$, where $\bar{t}(\kappa_c) = \pi/\mathrm{Re}(\sqrt{\kappa_c})$ (in particular $\bar{t}(\kappa_c)=+\infty$ if $\kappa_c \leq 0$). This function is the solution of a Cauchy problem of Riccati type with limit initial datum, that is
\begin{equation}
\dot{s} + \kappa_c + s^2 = 0, \qquad \lim_{t \to 0^+} s^{-1} = 0.
\end{equation}

\subsubsection{The ``sub-Rieman\-nian'' function} The second function $s_{\kappa_a,\kappa_b}: I \to \R$ does not appear in Riemannian geometry and depends on two real parameters $\kappa_a,\kappa_b$:
\begin{equation}\label{eq:sRiemanfunc}
s_{\kappa_a,\kappa_b}(t):= \frac{2}{t}\left(\frac{\sinc \left(2 \theta_- t\right)-\sinc \left(2 \theta _+ t\right)}{\sinc\left(\theta _- t\right)^2- \sinc\left(\theta _+ t\right)^2}\right), \qquad \theta_{\pm}= \frac{1}{2}(\sqrt{x+y} \pm \sqrt{x-y}), 
\end{equation}
where $\sinc(a) = \sin(a)/a$ is an entire function, and we have set
\begin{equation}
x = \frac{\kappa_b}{2}, \qquad y = \frac{\sqrt{4\kappa_a + \kappa_b^2}}{2}.
\end{equation}
Also~\eqref{eq:sRiemanfunc} is related with the solution of a matrix Cauchy problem of Riccati type, with limit initial datum (see Sections~\ref{s:comp} and \ref{s:proofs}). In this case, the maximal interval of definition is $I=(0,\bar{t}(\kappa_a,\kappa_b))$, and the time $\bar{t}(\kappa_a,\kappa_b)$ is called the \emph{first blow-up time}.
\begin{proposition}\label{p:stima-intro}
The first blow-up time $\bar{t}(\kappa_a,\kappa_b)$ of~\eqref{eq:sRiemanfunc} is bounded by
\begin{equation}\label{eq:stima-intro}
\bar{t}(\kappa_a,\kappa_b) \leq \frac{2\pi}{\mathrm{Re}(\sqrt{x+y}-\sqrt{x-y})}, \qquad x = \frac{\kappa_b}{2}, \qquad y= \frac{\sqrt{\kappa_b^2 + 4\kappa_a}}{2},
\end{equation}
where the r.h.s. of~\eqref{eq:stima-intro} is $+\infty$ if the denominator is zero. The equality holds if and only if $\kappa_a = 0$, in this case $\bar{t}(0,\kappa_b) = 2\pi/\sqrt{\kappa}_b$. In particular $\bar{t}(\kappa_a,\kappa_b)$ is finite if and only if
\begin{equation}
\begin{cases} \kappa_b \geq 0, & \\
\kappa_b^2 + 4\kappa_a > 0, &
\end{cases} \qquad\text{or} \qquad \begin{cases} \kappa_b < 0, & \\
\kappa_a > 0. &
\end{cases} \tag{$\star$} \label{eq:conditions}
\end{equation}
\end{proposition}

\subsection{Conjugate points}
Our first results are bounds for the first conjugate point along a sub-Rieman\-nian geodesic (i.e. the first critical value of the exponential map). 

\begin{theorem}[First conjugate time I]\label{t:conjI-intro} Let $(M,\distr,g)$ be a fat sub-Rieman\-nian manifold of type $(k,n)$. Let $\gamma(t)$ be a geodesic, with initial unit covector $\lambda$. Assume that
\begin{equation}
\Riccan^a(\lambda(t)) \geq (n-k)\kappa_a, \qquad \Riccan^b(\lambda(t)) \geq (n-k)\kappa_b,
\end{equation}
for some $\kappa_a,\kappa_b \in \R$ such that~\eqref{eq:conditions} are satisfied. Then the first conjugate time $t_*(\gamma)$ along the geodesic $\gamma$ is finite and
\begin{equation}
t_*(\gamma) \leq \bar{t}(\kappa_a,\kappa_b).
\end{equation}
\end{theorem}
\begin{theorem}[First conjugate time II]\label{t:conjII-intro} Let $(M,\distr,g)$ be a fat sub-Rieman\-nian manifold of type $(k,n)$, with $2k-n>1$. Let $\gamma(t)$ a geodesic, with initial unit covector $\lambda$. Assume that
\begin{equation}
\Riccan^c(\lambda(t)) \geq (2k-n-1)\kappa_c,
\end{equation}
for some $\kappa_c >0$. Then the first conjugate time $t_*(\gamma)$ along the geodesic $\gamma$ is finite and
\begin{equation}
t_*(\gamma) \leq \bar{t}(\kappa_c) = \frac{\pi}{\sqrt{\kappa_c}}.
\end{equation}
\end{theorem}
Theorem~\ref{t:conjII-intro} does not apply to ``maximally fat'' structures, namely when $n=2k-1$ (the maximal possible dimension for a given fat distribution of rank $k$). Globalizing the hypotheses, we obtain two sub-Rieman\-nian versions of the Bonnet-Myers theorem.
\begin{theorem}[Bonnet-Myers I]\label{t:bmI-intro} Let $(M,\distr,g)$ be a complete, fat sub-Rieman\-nian manifold of type $(k,n)$. Assume that, for any unit covector $\lambda$ 
\begin{equation}
\Riccan^a(\lambda) \geq (n-k) \kappa_a, \qquad \Riccan^b(\lambda) \geq (n-k) \kappa_b,
\end{equation}
for some $\kappa_a,\kappa_b \in \R$ satisfying~\eqref{eq:conditions}. Then the sub-Rieman\-nian diameter of $M$ is bounded by 
\begin{equation}
\diam(M) \leq \bar{t}(\kappa_a,\kappa_b).
\end{equation}
Moreover, $M$ is compact, and its fundamental group is finite.
\end{theorem}
\begin{theorem}[Bonnet-Myers II]\label{t:bmII-intro} Let $(M,\distr,g)$ be a complete, fat sub-Rieman\-nian manifold of type $(k,n)$, with $2k-n>1$. Assume that, for any unit covector $\lambda$
\begin{equation}
\Riccan^c(\lambda) \geq (2k -n-1)\kappa_c,
\end{equation}
for some $\kappa_c >0$. Then the sub-Rieman\-nian diameter of $M$ is bounded by 
\begin{equation}
\diam(M) \leq \bar{t}(\kappa_c) = \frac{\pi}{\sqrt{\kappa_c}}.
\end{equation}
Moreover, $M$ is compact, and its fundamental group is finite.
\end{theorem}

\subsection{Sub-Laplacian} For any function $f \in C^\infty(M)$, the \emph{horizontal gradient} $\grad(f) \in \Gamma(\distr)$ is, at each point, the horizontal direction of steepest slope of $f$, that is
\begin{equation}
g(\grad(f), X) = df(X), \qquad \forall X \in \Gamma(\distr).
\end{equation}
Fix any smooth volume form $\omega \in \Lambda^n M$ (or a density, if $M$ is not orientable). The \emph{divergence} of a smooth vector field is defined by the following identity
\begin{equation}
\mathcal{L}_X \omega = \dive_{\omega}(X) \omega, \qquad X \in \Gamma(TM),
\end{equation}
where $\mathcal{L}$ denotes the Lie derivative. We define the \emph{sub-Laplacian} $\Delta_\omega f := \dive_\omega(\grad(f))$ for any $f \in C^\infty(M)$. The sub-Laplacian is symmetric on the space $C^\infty_c(M)$ of smooth functions with compact support with respect to the $L^2(M,\omega)$ product:
\begin{equation}
\int_M f (- \Delta_\omega h) \omega = \int_M g(\grad(f), \grad(h)) \omega, \qquad \forall f,h \in C^\infty_c(M).
\end{equation}
If $(M,\d)$ is complete, then $\Delta_\omega$ is essentially self-adjoint on $C^\infty_c(M)$ and has a positive, smooth heat kernel \cite{strichartz,strichartzerrata}.

The sub-Laplacian is intrinsic if the choice of volume is. A natural choice is Popp volume \cite{montgomerybook,BR-Popp}. For the Hopf fibrations, it is proportional to the Riemannian volume of the corresponding round spheres, and the associated sub-Laplacian coincides with the one studied in \cite{BW-CR,BW-QHF}. See also \cite{ABGR-unimodular} for the case of unimodular 3D Lie groups. 

\subsection{Canonical volume derivative}
A new object appears in relation with the volume. To introduce it, consider a Riemannian manifold $(M,g)$ equipped with a volume $\omega$ (not necessarily the Riemannian one). Then, for all $\lambda \in T^*M \setminus H^{-1}(0)$, we define
\begin{equation}
\rho_\omega(\lambda) := \frac{\nabla_{\lambda^\sharp} \omega}{\omega}, 
\end{equation}
where $\sharp$ is the canonical musical isomorphism and $\nabla$ the Levi-Civita connection. Indeed $\rho_\omega$ is smooth and $\rho_\omega = 0$ if and only if $\omega$ is the Riemannian volume. 

The sub-Rieman\-nian generalization of $\rho_\omega : T^*M \setminus H^{-1}(0) \to \R$ plays an important role in the next theorems and we call it the \emph{canonical volume derivative} (see Section~\ref{s:jac}). In any contact Sasakian manifold equipped with Popp volume, as the ones considered in \cite{AL-3D-MCP,AL-3D-BisLap,LLZ-Sasakian-MCP,LL-Sasakian-BisLap}, $\rho_\omega =0$, similarly to the Riemannian case. We prove that this is true also in the $3$-Sasakian setting. This is not true in general.

\subsection{Sub-Riemannian distance} Assume from now on that $(M,\d)$ is complete. For any point $q_0 \in M$, let $r_{q_0}(\cdot):= \d(q_0,\cdot)$ be the sub-Rieman\-nian distance from $q_0$. By a by-now classical result \cite{agrachevsmooth,rifford2014sub}, $r_{q_0}$ is smooth on an open dense set (on a general sub-Rieman\-nian manifold). In addition, for fat structures, $\d: M \times M \to \R$ is locally Lipschitz in charts outside the diagonal and $r_{q_0}$ is smooth almost everywhere \cite{rifford2014sub,nostrolibro}.

\begin{theorem}[Sub-Laplacian comparison]
Let $(M,\distr,g)$ be a complete, fat sub-Rieman\-nian manifold of type $(k,n)$, equipped with a smooth volume (or density) $\omega$. Assume that for any unit covector $\lambda \in U^*_{q_0} M$
\begin{equation}
\begin{cases} \Riccan^a(\lambda(t))  \geq (n-k) \kappa_a,  \\
\Riccan^b(\lambda(t)) \geq (n-k) \kappa_b,  \\
\Riccan^c(\lambda(t)) \geq  (2k-n-1)\kappa_c,  
\end{cases} \quad \text{and} \qquad \rho_\omega(\lambda(t))  \leq \kappa_\omega,
\end{equation}
for some $\kappa_a,\kappa_b,\kappa_c, \kappa_\omega \in \R$. Then
\begin{equation}
\Delta_\omega r_{q_0}(q) \leq (n-k) s_{\kappa_a,\kappa_b}(r_{q_0}(q)) + (2k-n-1) s_{\kappa_c}(r_{q_0}(q)) + \kappa_\omega,
\end{equation}
almost everywhere.
\end{theorem}
This theorem can be improved for bounds that depend on the initial covector. If $r_{q_0}$ is smooth at $q$, then there exists a unique length-parametrized geodesic joining $q_0$ with $q$, and its initial covector is $\lambda_{q_0}^{q}= e^{-r_{q_0} \vec{H}}(d_q r_{q_0}) \in U_{q_0}^*M$, where $d_q$ denotes the differential at $q$.
\begin{theorem}[Sub-Laplacian comparison - weak statement] \label{t:lapl-intro}
Let $(M,\distr,g)$ be a complete, fat sub-Rieman\-nian manifold of type $(k,n)$, equipped with a smooth volume (or density) $\omega$. Assume that for any unit covector $\lambda \in U^*_{q_0} M$
\begin{equation}
\begin{cases} \Riccan^a(\lambda(t))  \geq (n-k) \kappa_a(\lambda),  \\
\Riccan^b(\lambda(t)) \geq (n-k) \kappa_b(\lambda),  \\
\Riccan^c(\lambda(t)) \geq  (2k-n-1)\kappa_c(\lambda), 
\end{cases} \quad \text{and} \qquad \rho_\omega(\lambda(t))  \leq \kappa_\omega(\lambda),
\end{equation}
for some $\kappa_a(\lambda),\kappa_b(\lambda),\kappa_c(\lambda), \kappa_\omega(\lambda) \in \R$, possibly depending on the initial covector. Then
\begin{equation}
\Delta r_{q_0} (q) \leq (n-k) s_{\kappa_a(\lambda_{q_0}^q),\kappa_b(\lambda_{q_0}^q)}(r_{q_0}(q)) + (2k-n-1) s_{\kappa_c(\lambda_{q_0}^q)}(r_{q_0}(q)) +  \kappa_\omega(\lambda_{q_0}^q),
\end{equation}
almost everywhere.
\end{theorem}

\subsection{3-Sasakian structures}

We pass now to applications. Following~\cite{blair}, a \emph{$3$-Sasakian structure} on a manifold $M$ of dimension $4d+3$, with $d\geq 1$, is a collection $\{\phi_\alpha,\eta_\alpha,\xi_\alpha,g\}_\alpha$, with $\alpha=I,J,K$, of three contact metric structures, where $g$ is a Riemannian metric, $\eta_\alpha$ is a one-form, $\xi_\alpha$ is the Reeb vector field and $\phi_\alpha : \Gamma(TM) \to \Gamma(TM)$ is given by
\begin{equation}
2g(X,\phi_\alpha Y) = d\eta(X,Y).
\end{equation}
The three structures are Sasakian, and $\phi_I,\phi_J,\phi_K$ satisfy quaternionic-like compatibility relations (see Section~\ref{s:3-Sas} for details). A natural sub-Riemannian structure is given by the restriction of the Riemannian metric $g$ to the distribution
\begin{equation}
\distr = \bigcap_{\alpha=I,J,K} \ker \eta_\alpha.
\end{equation}
The three Reeb vector fields $\xi_\alpha$ are an orthonormal triple, orthogonal to $\distr$. Finally, for these structures, Popp volume is proportional to the Riemannian one (Proposition~\ref{p:Popp3Sas}).

\begin{rmk}
Here we are interested in the sub-Riemannian structure $(M,\distr,g|_\distr)$. The Riemannian metric of the $3$-Sasakian structure on the directions transverse to $\distr$ is not relevant. 
\end{rmk}

\begin{example}[Quaternionic Hopf fibration] The quaternionic unit sphere is the real manifold of dimension $4d+3$
\begin{equation}
\mathbb{S}^{4d+3}=\left\lbrace q = (q_1,\ldots,q_{d+1}) \in \H^{d+1} \mid \|q\|=1 \right\rbrace,
\end{equation}
equipped with the standard round metric $g$. Let $\mathbf{n}$ be the inward unit normal vector of $\mathbb{S}^{4d+3} \subset \H^{d+1} \simeq \R^{4d+4}$. The multiplication by $I,J,K$ induces the endomorphisms $\Phi_\alpha : T \H^{d+1} \to T \H^{d+1}$, for $\alpha=I,J,K$. The three vectors $\xi_{\alpha}:= \Phi_\alpha \mathbf{n}$ are tangent to $\mathbb{S}^{4d+3}$. The endomorphisms $\phi_\alpha$ are given by the restrictions of the complex structures $\Phi_\alpha$ to $T\mathbb{S}^{4d+3}$ and the one forms $\eta_\alpha$ are the dual of the vectors $\xi_\alpha$ (w.r.t. the round metric). The sub-Rieman\-nian distribution $\distr$ is given by the orthogonal complement of $\spn\{\xi_I,\xi_J,\xi_K\}$ and the sub-Rieman\-nian metric is the restriction to $\distr$ of the Riemannian one. 
\end{example}


\begin{theorem}[Sub-Riemannian Ricci curvatures for $3$-Sasakian manifolds]\label{t:ricci-3-sas-intro}
Let $(M,\distr,g)$ be the sub-Riemannian structure of a $3$-Sasakian manifold of dimension $4d+3$. For any unit covector $\lambda \in U^*M$
\begin{align}
\mathfrak{Ric}^a(\lambda) & =  3\left(\tfrac{3}{4}\varrho^a(v)-\tfrac{7}{2}\|v\|^2-\tfrac{15}{8}\|v\|^4\right),\\
\mathfrak{Ric}^b(\lambda) & = 3(4 +5\|v\|^2), \\
\mathfrak{Ric}^c(\lambda) &  =(4d-4)(1+ \|v\|^2) ,
\end{align}
where $\|v\|^2 := v_I^2 + v_J^2 + v_K^2$ and $v_\alpha=\langle \lambda,\xi_\alpha\rangle$ for $\alpha=I,J,K$. Moreover the canonical volume derivative w.r.t. Popp volume vanishes, i.e. $\rho_\omega = 0$.
\end{theorem}
In the above theorem, $\varrho^a(v)$ is a sectional-like curvature invariant, given by
\begin{equation}
\varrho^a(v) := \sum_{\alpha=I,J,K} R^\nabla(\dot\gamma,Z_\alpha,Z_\alpha,\dot\gamma),
\end{equation}
where $R^\nabla$ is the Riemannian curvature of the $3$-Sasakian structure, $\dot\gamma$ is the tangent vector of the sub-Riemannian geodesic associated with $\lambda$, and the vectors $Z_I,Z_J,Z_K \in \distr$ are
\begin{equation*}
Z_I := (v_J \phi_K  - v_K \phi_J) \dot\gamma, \qquad Z_J := (v_K \phi_I  - v_I \phi_K) \dot\gamma, \qquad Z_K := (v_I \phi_J  - v_J \phi_I) \dot\gamma.
\end{equation*}
\
\begin{rmk} Observe that $\mathfrak{Ric}^a$, the most complicated of the sub-Rieman\-nian curvatures, is not even a quadratic form as a function of the covector $\lambda$. The functions $v_\alpha: T^*M \to \R$  are prime integrals of the sub-Riemannian geodesic flow (Lemma~\ref{l:costanza}), hence $\Riccan^a$ is the only curvature that can depend on time, when evaluated along the extremal $\lambda(t)$. This is dramatically different w.r.t. the Sasakian case where $\mathfrak{Ric}^a =0$ (see \cite{LLZ-Sasakian-MCP}).
\end{rmk}

\subsection{Sharp estimates for the sub-Riemannian diameter}
Any complete $3$-Sasakian structure is Einstein, with constant scalar curvature equal to $(4d+2)(4d+3)$ (see Theorem~\ref{t:BGM}). In particular, by the classical Bonnet-Myers theorem, it is compact with Riemannian diameter bounded by $\pi$. Nevertheless, this gives no information on the \emph{sub-Riemannian} diameter that, a priori, could be larger. When $d>1$, Theorem~\ref{t:bmII-intro} yields the following.
\begin{corollary}\label{c:3sasBM-1}
Let $(M,\distr,g)$ be the sub-Riemannian structure of a $3$-Sasakian manifold $M$, of dimension $4d+3$, with $d>1$. The sub-Riemannian diameter is not larger than $\pi$.
\end{corollary}
Moreover, Theorem~\ref{t:conjII-intro}, with $\kappa_c(\lambda) = 1+ \|v\|^2$, yields the following.
\begin{corollary}\label{c:3sas-conj}
Let $(M,\distr,g)$ be the sub-Riemannian structure of a $3$-Sasakian manifold $M$, of dimension $4d+3$, with $d>1$. Then any sub-Riemannian geodesic with initial covector $\lambda \in U^*M$ has a conjugate point at distance $t_*(\lambda) \leq \tfrac{\pi}{\sqrt{1+\|v\|^2}}$, where $\|v\|$ is as in Theorem~\ref{t:ricci-3-sas-intro}.
\end{corollary}
Theorem~\ref{t:bmI-intro} does not apply for $d=1$, as a covector-independent lower bound is not possible. However, Theorem~\ref{t:conjI-intro} and careful estimates give the maximal conjugate distance.
\begin{proposition}\label{p:3sasBM-2}
Let $(M,\distr,g)$ be the sub-Riemannian structure of a $3$-Sasakian manifold $M$ of dimension $4d+3$, with $d \geq 1$. Assume that, for all $q \in M$ and any vector $X \in \distr_q$
\begin{equation}\label{eq:boundrhoa}
\Sec(X,Y) \geq K \geq -1, \qquad \forall\, Y \in \spn\{\phi_I X,\phi_J X,\phi_K X\},
\end{equation}
where $\Sec$ is the Riemannian sectional curvature of the $3$-Sasakian structure. Then the sub-Riemannian diameter is not larger than $\pi$.
\end{proposition}
For any quaternionic Hopf fibration (QHF in the following), Proposition~\ref{p:3sasBM-2} applies with $K =1$, and we obtain $\diam(\mathbb{S}^{4d+3}) \leq \pi$. For any $d \geq 1$, the sub-Riemannian distance of the QHF has been computed in \cite{BW-QHF}, using Ben Arous and L\'eandre formulas and heat kernel expansions, and the sub-Riemannian diameter is equal to $\pi$. Thus our results are sharp.

\subsection*{Open problem} 
The Riemannian diameter of any $3$-Sasakian manifold of dimension $4d+3$ is bounded by $\pi$. Corollary~\ref{c:3sasBM-1} extends this universal bound to the sub-Riemannian diameter, provided that $d>1$. For the case $d=1$, Proposition~\ref{p:3sasBM-2} requires some curvature assumptions that, a priori, might be violated. However, it would be surprising, for us, to find an example of $7$-dimensional $3$-Sasakian manifold with sub-Riemannian diameter larger than $\pi$. Thus, we close with the following question:
\begin{center}
\emph{Is it true that \emph{any} $3$-Sasakian manifold has sub-Riemannian diameter bounded by $\pi$?}
\end{center}

%
%

\subsection{Comparison with recent literature}

The curvature employed in this paper arises in a general setting, as a complete set of invariants of the so-called Jacobi curves. It has been introduced by Agrachev and Gamkrelidze in \cite{agrafeedback}, Agrachev and Zelenko in \cite{agzel1,agzel2} and extended by Zelenko and Li in \cite{lizel}. A closely related approach to curvature, valid for a general class of variational problems, is discussed in \cite{ABR-curvature} and in \cite{ABR-curvaturecontact} for contact structures.

This paper is not the first one to discuss comparison-type results on sub-Rieman\-nian manifolds, but it has been inspired by some recent works. The first results for the number of conjugate points along a given geodesics under sectional-type curvature bounds are in~\cite{lizel2}, for corank $1$ structures with transverse symmetries. Comparison theorems based on matrix Riccati techniques appear in \cite{AL-3D-MCP} (with applications to the measure contraction properties of 3D contact sub-Rieman\-nian manifolds) and in the subsequent series of papers~\cite{AL-3D-BisLap,LL-Sasakian-BisLap,LLZ-Sasakian-MCP} for Sasakian sub-Rieman\-nian structures.

The canonical Ricci curvatures, as partial traces of the canonical curvature, have been introduced in \cite{BR-comparison}. The comparison results obtained here for fat sub-Rieman\-nian structures are based on the same machinery. Nevertheless, some key technical results are proved here in a more geometrical fashion. Moreover, the explicit form of the ``bounding functions'' $s_{\kappa_a,\kappa_b}(t)$ is fundamental for proving quantitative results and it is obtained here for the first time.

The canonical curvature does not arise in relation with some linear connection, but with a non-linear Ehresmann one \cite{BR-connection,lizel}. Non-linear connections are not associated with a classical covariant derivative and thus this approach lacks the power of standard tensorial calculus.

Sometimes, a sub-Rieman\-nian structure comes with a ``natural'' Riemannian extension and one might want to express the sub-Rieman\-nian curvatures in terms of the Levi-Civita connection of the extension. The actual computation is a daunting task, as in doing this we are writing an intrinsically sub-Rieman\-nian object (the canonical Ricci curvatures) in terms of an extrinsic Riemannian structure. This task, however, is important, as it provides models in which the curvature is explicit (just as the Riemannian space forms). Results in this sense are interesting \emph{per se} and have been obtained, so far, for corank $1$ structures with symmetries \cite{lizel}, contact Sasakian structures \cite{LL-Sasakian-BisLap} and contact structures~\cite{ABR-curvaturecontact}. Our results are the first explicit expressions for corank greater than $1$.

An alternative approach to curvature in sub-Riemannian geometry is the one based on the so-called generalized Curvature Dimension (CD) inequality, introduced by Baudoin and Garofalo in \cite{garofalob}. These techniques can be applied to sub-Rieman\-nian manifolds with transverse symmetries. In \cite{baudoincontact}, Baudoin and Wang generalize these results removing the symmetries assumption for contact structures. In \cite{BKW-Weitzen} the same techniques are further generalized to Riemannian foliations with totally geodesic leaves. This class include the QHF, and our study has been motivated also by these works. See \cite{garobonneinequalities,baudoinmunivegarofalo,BK-Lich-Obata} for other comparison-type results following from the generalized CD condition.

The universal estimate $\diam(M) \leq \pi$ for the sub-Riemannian diameter of $3$-Sasakian structures of dimension $4d+3$, with $d>1$, is perhaps the most surprising result of this paper. As we already mentioned, the estimate is sharp for the QHF, whose explicit diameter has been obtained in \cite[Remark 2.15]{BW-QHF}. The same estimate holds for the sub-Riemannian structure on the complex Hopf fibration of $\mathbb{S}^{2d+1}$, as proved in \cite[Remark 3.11]{BW-CR}, but clearly it does not hold for general Sasakian structures.

Very recently, in \cite{BI-QC}, Theorem~\ref{t:bmII-intro} has been applied to quaternionic contact structures of dimension $4d+3$, with $d>1$, to yield Bonnet-Myers type results under suitable assumptions on the curvature associated with the Biquard connection.

Finally, we mention that estimates for \emph{Riemannian} diameter of Sasakian structures have been obtained in \cite{HS-remarks,N-bound}, under lower bounds on the transverse part of the Ricci curvature. Furthermore, Bonnet-Myers type theorems for the Riemannian structure of quaternionic contact structures appeared recently in~\cite{Hladky-qc}.

\subsection*{Structure of the paper}
In Section~\ref{s:jac} we present the theory of sub-Riemannian Jacobi fields and the curvature in the sense of Agrachev-Li-Zelenko. In Section~\ref{s:comp} we discuss the matrix Riccati comparison theory that we need in the rest of the paper. Section~\ref{s:proofs} is dedicated to the proofs of the results stated in Section~\ref{s:intro}. In Section~\ref{s:3-Sas} and \ref{s:CanFrame} we discuss the sub-Riemannian structure of $3$-Sasakian manifolds and we compute their sub-Riemannian curvature.


\section{Sub-Riemannian Jacobi equations and curvature} \label{s:jac}

\subsection{Jacobi equation revisited}
For any vector field $V(t)$ along an extremal $\lambda(t)$ of the sub-Riemannian Hamiltonian flow, a dot denotes the Lie derivative in the direction of $\vec{H}$:
\begin{equation}
\dot{V}(t) := \left.\frac{d}{d\eps}\right|_{\eps=0} e^{-\eps \vec{H}}_* V(t+\eps).
\end{equation}
A vector field $\J(t)$ along $\lam(t)$ is called a \emph{sub-Riemannian Jacobi field} if it satisfies 
\begin{equation}\label{eq:defJF}
\dot{\J} = 0.
\end{equation}
The space of solutions of \eqref{eq:defJF} is a $2n$-dimensional vector space. The projections $\pi_{*}\J(t)$ are vector fields on $M$ corresponding to one-parameter variations of $\g(t)=\pi(\lam(t))$ through geodesics; in the Riemannian case, they coincide with the classical Jacobi fields.

We write \eqref{eq:defJF} using the symplectic structure $\sigma$ of $T^{*}M$.  First, observe that on $T^*M$ there is a natural smooth sub-bundle of Lagrangian\footnote{A subspace $L \subset \Sigma$ of a symplectic vector space $(\Sigma,\sigma)$ is Lagrangian if $\dim L = \dim\Sigma/2$ and $\sigma|_{L} = 0$.} spaces:
\begin{equation}
\ver_{\lambda} := \ker \pi_*|_{\lambda} = T_\lambda(T^*_{\pi(\lambda)} M) \subset T_{\lambda}(T^*M).
\end{equation}
We call this the \emph{vertical subspace}. Then, pick a Darboux frame $\{E_i(t),F_i(t)\}_{i=1}^{n}$ along $\lambda(t)$.  It is natural to assume that $E_1,\ldots,E_n$ belong to the vertical subspace. To fix the ideas, one can think at the frame $\partial_{p_i}|_{\lambda(t)},\partial_{x_i}|_{\lambda(t)}$ induced by coordinates $(x_1,\ldots,x_n)$ on $M$. In terms of this frame, $\J(t)$ has components $(p(t),x(t)) \in \R^{2n}$:
\begin{equation}
\J(t) = \sum_{i=1}^n p_{i}(t) E_{i}(t) + x_{i}(t) F_{i}(t).
\end{equation}
The elements of the frame satisfy\footnote{The notation of~\eqref{eq:Jacobiframe} means that $\dot{E}_i = \sum_{j=1}^n A(t)_{ij} E_j -B(t)_{ij} F_j$, and similarly for $\dot{F}_i$.} 
\begin{equation}\label{eq:Jacobiframe}
\frac{d}{dt}\begin{pmatrix}
E \\
F
\end{pmatrix} = 
\begin{pmatrix}
A(t) & -B(t) \\
R(t) & -A(t)^*
\end{pmatrix} \begin{pmatrix}
E\\
F
\end{pmatrix},
\end{equation}
for some smooth families of $n\times n$ matrices $A(t),B(t),R(t)$, where $B(t) = B(t)^*$ and $R(t)= R(t)^*$. The special structure of~\eqref{eq:Jacobiframe} follows from the fact that the frame is Darboux, that is
\begin{equation}
\sigma(E_i,E_j) = \sigma(F_i,F_j) = \sigma(E_i,F_j) -\delta_{ij} = 0, \qquad i,j=1,\ldots,n.
\end{equation}
For any bi-linear form $\mathcal{B}: V\times V \to \R$ and $n$-tuples $v,w \in V$ let $\mathcal{B}(v,w)$ denote the matrix $\mathcal{B}(v_i,w_j)$. With this notation
\begin{equation}
B(t) = \sigma(\dot{E},E)|_{\lambda(t)} = 2H(E,E)|_{\lambda(t)} \geq 0,
\end{equation}
where we identified $\ver_{\lambda(t)} \simeq T_{\gamma(t)}^*M$ and the Hamiltonian with a symmetric bi-linear form on fibers. In the Riemannian case, $B(t) > 0$. Finally, the components $(p(t),x(t))$ of $\mathcal{J}(t)$ satisfy
\begin{equation}
\frac{d}{dt} \begin{pmatrix}\label{eq:Jacobicoord}
p \\ x
\end{pmatrix} = \begin{pmatrix} - A(t)^* & -R(t) \\ B(t) & A(t)
\end{pmatrix} \begin{pmatrix}
p \\ x
\end{pmatrix}.
\end{equation}	
We want to choose a suitable frame to simplify~\eqref{eq:Jacobicoord} as much as possible.

\subsection{The Riemannian case}
It is instructive to study first the Riemannian setting. Let $f_1,\ldots,f_n$ be a parallel transported frame along the geodesic $\gamma(t)=\pi(\lambda(t))$.  Let $h_i:T^*M \to \R$, defined by $h_i(\lambda):= \langle \lambda, f_i\rangle$. They define coordinates on each fiber and, in turn, the vectors $\partial_{h_{i}}$. We define a moving frame along the extremal $\lambda(t)$ as follows
\begin{equation}
E_i:=\partial_{h_i}, \qquad F_i := -\dot{E}_i, \qquad i=1,\ldots,n.
\end{equation}
One recovers the original parallel transported frame by projection, namely $\pi_* F_i(t) = f_i|_{\g(t)}$. In the following, $\mathbbold{1}$ and $\mathbbold{0}$ denote the identity and zero matrices of the appropriate dimension.
\begin{proposition}\label{p:riemcan}
The smooth moving frame $\{E_i(t),F_i(t)\}_{i=1}^n$ along $\lambda(t)$ satisfies:
\begin{itemize}
\item[(i)] $\spn\{E_1(t),\ldots,E_n(t)\}=\ver_{\lambda(t)}$.
\item[(ii)] It is a Darboux basis, namely
\[
\sigma(E_i,E_j) = \sigma(F_i,F_j) = \sigma(E_i,F_j) - \delta_{ij} = 0, \qquad i,j=1,\ldots,n.
\]
\item[(iii)] The frame satisfies the \emph{structural equations}
\begin{equation}
\frac{d}{dt} \begin{pmatrix}
E \\
F
\end{pmatrix} = \begin{pmatrix}
\mathbbold{0} & - \mathbbold{1} \\
R(t) & \mathbbold{0}
\end{pmatrix} \begin{pmatrix}
E \\
F
\end{pmatrix},
\end{equation}
for some smooth family of $n\times n$ symmetric matrices $R(t)$.
\end{itemize}
If $\{\wt{E}_i,\wt{F}_j\}_{i=1}^n$ is another smooth moving frame along $\lambda(t)$ satisfying (i)-(iii), for some symmetric matrix $\wt{R}(t)$ then there exists a constant, orthogonal matrix $O$ such that 
\begin{equation}\label{eq:orthonormal}
\wt{E}(t) = O E(t), \qquad  \wt{F}(t) = O F(t), \qquad \wt{R}(t) = O R(t) O^*. 
\end{equation}
\end{proposition}
As a consequence, the matrix $R(t)$ gives a well defined operator $\mathfrak{R}_{\lam(t)}:T_{\gamma(t)}M \to T_{\gamma(t)}M$
\begin{equation}
\Rcan_{\lam(t)} v := \sum_{i,j=1}^n R_{ij}(t) v_{j} f_i|_{\gamma(t)} , \qquad v = \sum_{i=1}^n v_i f_i|_{\gamma(t)}.
\end{equation}
With a routine but long computation (for instance, see \cite[Appendix C]{BR-comparison}) one checks that 
\begin{equation}\label{eq:trc}
\Rcan_{\lam(t)}v = R^\nabla(v,\dot{\gamma})\dot{\gamma}, \qquad v \in T_{\gamma(t)}M,
\end{equation}
where $R^\nabla(X,Y)Z = \nabla_X \nabla_Y Z - \nabla_Y\nabla_X Z - \nabla_{[X,Y]} Z$ is the Riemannian curvature tensor associated with the Levi-Civita connection $\nabla$. Then, in the Jacobi equation \eqref{eq:Jacobicoord}, one has $A(t) =\mathbbold{0}$, $B(t) = \mathbbold{1}$, and the only non-trivial block is the curvature $R(t)$:
\begin{equation}
\dot x=p, \qquad \dot{p} = -R(t) x,
\end{equation}
that is the classical Riemannian Jacobi equation $\ddot{x} + R(t) x = 0$.

\subsection{The fat sub-Riemannian case}\label{sec:sasakiannormalframe}
The normal form of the sub-Riemannian Jacobi equation has been first studied by Agrachev-Zelenko in \cite{agzel1,agzel2} and subsequently completed by Zelenko-Li in \cite{lizel}, in the general setting of curves in the Lagrange Grassmannian. A dramatic simplification, analogue to the Riemannian one, cannot be achieved in general. Nevertheless, it is possible to find a normal form of \eqref{eq:Jacobicoord} where the matrices $A(t)$ and $B(t)$ are constant. The general result, in the language of Proposition~\ref{p:riemcan}, can be found in \cite{BR-connection}. Here we give an ad-hoc statement for fat sub-Riemannian structures.

\textbf{Notation.} It is convenient to split the set of indices $1,\ldots,n$ in the following subsets:
\begin{equation}
\underbrace{1,\ldots,n-k}_{a}, \underbrace{n-k+1,\ldots,2n-2k}_{b}, \underbrace{2n-2k+1,\ldots,n}_{c}.
\end{equation}
The cardinality of the sets of indices are $|a| = |b| = n-k$, $|c| = 2k-n$. 
Accordingly, we write any $n \times n$ matrix $L$ in block form, as follows
\begin{equation}\label{eq:decomposition}
L = \begin{pmatrix}
L_{aa} & L_{ab} & L_{ac} \\
L_{ba} & L_{bb} & L_{bc} \\
L_{ca} & L_{cb} & L_{cc}
\end{pmatrix},
\end{equation}
where $L_{\mu\nu}$, for $\mu,\nu = a,b,c$ is a matrix of dimension $|\mu| \times |\nu|$. Analogously, we split $n$-tuples $Z=(Z_a,Z_b,Z_c)$. Accordingly, for any bi-linear form $Q$, the notation $Q(Z_\mu,Z_\nu)$, with $\mu,\nu =a,b,c$ denotes the associated $|\mu|\times |\nu|$ matrix. 
\begin{rmk}
This splitting is related to the fact that the Lie derivative in the direction of a fixed $X \in \distr$ induces a well defined, surjective linear map $\mathcal{L}_X : \distr_q \to T_q M/\distr_q$. It has a $n-k$-dimensional image (the ``$a$'' space), a $2k-n$-dimensional kernel (the ``$c$'' space, and the orthogonal complement of the latter in $\distr_q$ is a $n-k$-dimensional space (the ``$b$'' space).
\end{rmk}
\begin{theorem}\label{p:can}
Let $\lambda(t)$ be an extremal of a fat sub-Riemannian structure. There exists a smooth moving frame along $\lambda(t)$
\begin{equation}
E(t)  = (E_a(t),E_b(t),E_c(t))^*, \qquad F(t)  = (F_a(t),F_b(t),F_c(t))^*,
\end{equation}
such that the following holds true for any $t$:
\begin{itemize}
\item[(i)] $\spn\{E_a(t),E_b(t),E_c(t)\} = \mathcal{V}_{\lambda(t)}$.
\item[(ii)] It is a Darboux basis, namely
\begin{equation}
\sigma(E_\mu,E_\nu) = \sigma(F_\mu,F_\nu)= \sigma(E_\mu,F_\nu) - \delta_{\mu\nu} \mathbbold{1} = 0, \qquad \mu,\nu =a,b,c.
\end{equation}
\item[(iii)] The frame satisfies the \emph{structural equations}
\begin{equation}
\frac{d}{dt}\begin{pmatrix}
\dot{E} \\
\dot{F}
\end{pmatrix} = \begin{pmatrix}
 A & - B \\
R(t) & - A^*
\end{pmatrix} \begin{pmatrix}
E \\
F
\end{pmatrix},
\end{equation}
where $A$, $B$ are constant, $n \times n$ block matrices defined by 
\begin{equation}
A := \begin{pmatrix}
\mathbbold{0} & \mathbbold{1} & \mathbbold{0} \\ \mathbbold{0} & \mathbbold{0} & \mathbbold{0} \\ \mathbbold{0} & \mathbbold{0} & \mathbbold{0}
\end{pmatrix}, \qquad B := \begin{pmatrix}
\mathbbold{0} & \mathbbold{0} & \mathbbold{0} \\ \mathbbold{0} & \mathbbold{1} & \mathbbold{0} \\ \mathbbold{0} & \mathbbold{0} &\mathbbold{1}
\end{pmatrix}.
\end{equation}
Finally $R(t)$ is a $n \times n $ smooth family of symmetric matrices of the form
\begin{equation}
R(t) = \begin{pmatrix}
R_{aa}(t) & R_{ab}(t) & R_{ac}(t) \\
R_{ba}(t) & R_{bb}(t) & R_{bc}(t) \\
R_{ca}(t) & R_{cb}(t) & R_{cc}(t)
\end{pmatrix},
\end{equation}
with the additional condition
\begin{equation}\label{eq:normalcond}
R_{ab}(t) = - R_{ab}(t)^*.
\end{equation}
\end{itemize}

If $\{\wt{E}(t),\wt{F}(t)\}$ is another frame that satisfies (i)-(iii) for some matrix $\wt{R}(t)$, then there exists a constant $n \times n$ orthogonal matrix $O$ that preserves the structural equations (i.e.  $O A O^* = A$, $OBO^* = B$) and
\begin{equation}
\wt{E}(t) = O E(t), \qquad \wt{F}(t) = O F(t), \qquad \wt{R}(t) = O R(t) O^*.
\end{equation}
\end{theorem}

\subsection{Invariant subspaces and curvature}\label{s:invariantspaces}
The projections $f_\mu(t):= \pi_* F_\mu(t)$, with $\mu = a,b,c$, define a smooth frame along $T_{\gamma(t)}M$. The uniqueness part of Theorem~\ref{p:can} implies that this frame is unique up to a constant orthogonal transformation
\begin{equation}
O = \begin{pmatrix}
U_1 & \mathbbold{0} & \mathbbold{0} \\
\mathbbold{0} & U_1 & \mathbbold{0} \\
\mathbbold{0} & \mathbbold{0} & U_2
\end{pmatrix}, \qquad U_1 \in \mathrm{O}(n-k), \quad U_2 \in \mathrm{O}(2k-n).
\end{equation}
Thus, the following definitions are well posed.

\begin{definition}
The \emph{canonical splitting} of $T_{\gamma(t)} M$ is
\begin{equation}
T_{\gamma(t)}M = S_{\gamma(t)}^{a} \oplus S_{\gamma(t)}^{b} \oplus S_{\gamma(t)}^{c},
\end{equation}
where the \emph{invariant subspaces} are defined by
\begin{align}
S_{\gamma(t)}^a & := \spn\{ f_a \}, \qquad \dim S_{\gamma(t)}^a = n-k,\\
S_{\gamma(t)}^b & := \spn\{ f_b \}, \qquad \dim S_{\gamma(t)}^b = n-k,\\
S_{\gamma(t)}^c & := \spn\{ f_c\}, \qquad \dim S_{\gamma(t)}^c = 2k-n,
\end{align}
\end{definition}
\begin{definition}
The \emph{canonical curvature} along the extremal $\lambda(t)$ is the operator $\Rcan_{\lambda(t)}: T_{\gamma(t)} M \to T_{\gamma(t)} M$ that, in terms of the basis $f_a,f_b,f_c$ is represented by the matrix $R(t)$.
\end{definition}
For $\mu,\nu= a,b,c$, we denote by $\mathfrak{R}_{\lambda(t)}^{\mu\nu} : S^\mu_{\gamma(t)} \to S^\nu_{\gamma(t)}$ the restrictions of the canonical curvature to the appropriate invariant subspace.
\begin{definition}
The \emph{canonical Ricci curvatures} along the extremal $\lambda(t)$ are the partial traces
\begin{equation}
\Riccan^\mu(\lambda(t)):= \trace \left( \Rcan_{\lambda(t)}^{\mu\mu} : S_{\gamma(t)}^\mu \to S_{\gamma(t)}^\mu \right), \qquad \mu = a,b,c.
\end{equation}
\end{definition}
\begin{rmk}\label{rmk:traslation}
If $\{E(t),F(t)\}$ is a canonical frame along $\lambda(t)$ with initial covector $\lambda$ and curvature matrix $R(t)$, then $\{E(t+\tau),F(t+\tau)\}$ is a canonical frame along $\lambda_\tau(t)=\lambda(t+\tau)$ with initial covector $\lambda_\tau = \lambda(\tau)$ and curvature matrix $R(t+\tau)$. Therefore $\Riccan^\mu(\lambda_\tau(t)) = \Riccan^\mu(\lambda(t+\tau))$. For this reason, it makes sense to define $\Riccan^\mu : U^*M \to \R$, for any initial unit covector $\lambda \in U^*M$, as $\Riccan^\mu(\lambda):= \Riccan^\mu(\lambda(0))$. In particular, the hypothesis $\Riccan^\mu(\lambda) \geq \kappa$ for all $\lambda \in U^*M$ implies that for any extremal $\lambda(t) = e^{t\vec{H}}(\lambda)$ one has $\Riccan^\mu(\lambda(t)) \geq \kappa$.
\end{rmk}

\begin{rmk}\label{rmk:dimensionalreduction}
One can always choose a canonical frame in such a way that one of the $f_c(t)$'s (e.g., the last one) is the tangent vector of the associated geodesic $\dot\gamma(t)$, and lies in the kernel of the curvature operator. Thus, the $(2k-n)\times (2k-n)$ matrix $R_{cc}(t)$ splits further as
\begin{equation}
R_{cc}(t) = \begin{pmatrix}
R_{cc}^\prime(t) & 0 \\
0 & 0
\end{pmatrix},
\end{equation}
where $R_{cc}^\prime(t)$ is a $(2k-n-1)\times (2k-n-1)$ block. Moreover, $\Riccan^c(\lambda(t)) = \trace(R^\prime_{cc}(t))$.
\end{rmk}
\begin{rmk}\label{rmk:homogeneityproperties}
Let $\lambda \in T^*M \setminus H^{-1}(0)$ be a covector with corresponding extremal $\lambda(t)= e^{t\vec{H}}(\lambda)$. Let $\alpha >0$ and consider the rescaled covector $\alpha \lambda$, with the corresponding extremal $\lambda^\alpha(t) = e^{t\vec{H}}(\alpha \lambda)$. Then the Ricci curvatures have the following homogeneity properties
\begin{align}
\Riccan^a(\lambda^\alpha(t)) & = \alpha^4 \Riccan^a(\lambda(\alpha t)), \\
\Riccan^b(\lambda^\alpha(t)) & = \alpha^2 \Riccan^a(\lambda(\alpha t)), \\
\Riccan^c(\lambda^\alpha(t)) & = \alpha^2 \Riccan^a(\lambda(\alpha t)).
\end{align}
The proof follows from more general homogeneity properties of $\mathfrak{R}$ (see \cite[Theorem 4.7]{BR-connection}).
\end{rmk}

\begin{definition}
Let $\omega \in \Lambda^n M$ be a smooth volume form (or density, if $M$ is not orientable). The \emph{canonical volume derivative} $\rho_\omega: T^*M \setminus H^{-1}(0) \to \R$ is
\begin{equation}
\rho_\omega(\lambda) := \left.\frac{d}{dt}\right|_{t=0} \log |\omega(f_a(t),f_b(t),f_c(t))|, \qquad \lambda \in T^*M  \setminus H^{-1}(0).
\end{equation}
where $f_a(t),f_b(t),f_c(t)$ is a canonical frame associated with the extremal $\lambda(t)=e^{t\vec{H}}(\lambda)$.
\end{definition}
\begin{rmk}\label{r:homgeneity-vol}
The same construction, in the Riemannian setting, gives $\rho_\omega(\lambda) = \frac{\nabla_{\lambda^\sharp}\omega}{\omega}$, where $\nabla$ is the Levi-Civita connection. From the homogeneity properties of the canonical frame (see \cite[Proposition 4.9]{BR-connection}), it follows that $\rho_\omega(\alpha \lambda) = \alpha \rho_\omega(\lambda)$ for all $\alpha > 0$.
\end{rmk}
In \cite{ABP-distortion}, the above definition has been generalized to any sub-Riemannian structure, provided that the extremal $\lambda(t)$ satisfies some regularity conditions which, in the fat case, are verified. We notice that the definition of $\rho_\omega$ in \cite{ABP-distortion}, which the authors call \emph{volume geodesic derivative}, does not require the canonical frame.

\section{Matrix Riccati comparison theory}\label{s:comp}

The next lemma is immediate and follows from the definition of conjugate time.
\begin{lemma}\label{l:conjvert}
Let $\gamma(t)$ be a sub-Riemannian geodesic, associated with an extremal $\lambda(t)$. A time $t_*>0$ is conjugate if and only if there exists a Jacobi field $\J(t)$ along $\lambda(t)$ such that
\begin{equation}
\pi_*\J(0) = \pi_* \J(t_*) = 0,
\end{equation}
or, equivalently, $\J(0) \in \ver_{\lambda(0)}$, and $\J(t_*) \in \ver_{\lambda(t_*)}$. If $t_*$ is the first conjugate time along $\gamma$, any Jacobi field $\J(t)$ along $\lambda(t)$ is transverse to $\ver_{\lambda(t)}$ for all $t \in (0,t_*)$.
\end{lemma}

Choose the canonical moving frame of Theorem~\ref{p:can} along $\lambda(t)$, and consider the Jacobi fields $\J_i(t) \simeq (p_i(t),x_i(t))$, for $i=1,\ldots,n$, specified by the initial conditions
\begin{equation}
p_i(0) = (0,\ldots,1,\ldots,0)^*, \qquad x_i(0) = (0,\ldots,0)^*,
\end{equation}
where the $1$ is in the $i$-th position. We collect the column vectors $\J_i(t)$ in a $2n\times n$ matrix:
\begin{equation}
\bm{\J}(t) := [\J_1(t),\,\cdots,\J_n(t)] = \begin{pmatrix}
M(t) \\
N(t)
\end{pmatrix},
\end{equation}
where $M(t)$ and $N(t)$ are smooth families of $n\times n$ matrices. From~\eqref{eq:Jacobicoord}, we obtain
\begin{equation}\label{eq:Jacobimatrix}
\frac{d}{dt}\begin{pmatrix}
 M(t) \\
 N(t)
\end{pmatrix} = \begin{pmatrix}
-A^* & -R(t) \\
B & A
\end{pmatrix} \begin{pmatrix}
M(t) \\
N(t)
\end{pmatrix}, \qquad M(0) =\mathbbold{1}, \quad N(0) = \mathbbold{0}.
\end{equation}
Observe that, in general, a Jacobi field $\sum p_i(t) E_i(t) +x_i(t)F_i(t) \in \ver_{\lambda(t)}$ if and only if $x(t) =0$. Thus (the rows of) $\bm{\J}(t)$ describe the $n$-dimensional subspace of Jacobi fields $\J(t)$ with initial condition $\J(0) \in \ver_{\lambda(0)}$. Hence, the first conjugate time $t_*$ is precisely the smallest positive time such that $\det N(t_*) = 0$.

The $n\times n$ matrix $V(t):=M(t)N(t)^{-1}$ is well defined and smooth for all $t \in (0,t_*)$. One can check that it is a solution of the following Cauchy problem with limit initial datum
\begin{equation}\label{eq:riccati}
\dot{V} + A^* V + V A +R(t)+ V B V =0, \qquad \lim_{t \to 0^+} V^{-1} = 0,
\end{equation}
in the sense that $V(t)$ is invertible for small $t>0$ and $\lim_{t \to 0^+} V^{-1} =0$.

\subsection{The Matrix Riccati equation}
The nonlinear ODE~\eqref{eq:riccati} is called \emph{matrix Riccati equation}. An extensive literature on comparison theorems is available, see for example \cite{Roydencomp,abou2003matrix,esch}. Comparison theorems for solutions of~\eqref{eq:riccati} with limit initial datum are considered, to our best knowledge, only in \cite[Appendix A]{BR-comparison}. We take from there the results that we need.

\textbf{Assumptions.} In the following, $A,B$ are any pair of $n\times n$ matrices satisfying\footnote{Condition~\eqref{eq:Kalman} is called \emph{Kalman condition} in geometric control theory~\cite{Coronbook,Agrachevbook,Jurdjevicbook}.}
\begin{equation}\label{eq:Kalman}
\spn\{B,AB,\ldots,A^m B\} = \R^n,
\end{equation}
with $B \geq 0$ and $Q(t) = Q(t)^*$ is any smooth family of $n\times n$ matrices defined for $t \in [0,+\infty)$.
\begin{rmk}
The matrices $A$ and $B$ that appear in the Cauchy problem~\eqref{eq:riccati} for the case of fat sub-Riemannian structures (defined in Theorem~\ref{p:can}) verify~\eqref{eq:Kalman} with $m=1$.
\end{rmk}

\begin{lemma}[Well posedness]\label{l:wellpos}
The Cauchy problem with limit initial condition
\begin{equation}\label{eq:wellpos}
\dot{V} + A^* V + V A + Q(t) + VB V = 0, \qquad \lim_{t \to 0^+} V^{-1} = 0,
\end{equation}
is well posed, in the sense that it admits a smooth solution, invertible for small $t>0$, such that $\lim_{t \to 0^+} V^{-1} =0$. The solution is unique on a maximal interval of definition $I = (0,\bar{t})$ and symmetric. In addition, $V(t) > 0$ for small $t>0$.
\end{lemma}
The extrema of the interval of definition $(0,\bar{t})$ are characterized by the blow-up of $V(t)$. To be precise, we say that a one-parameter family $V(t)$ of $n\times n$ symmetric matrices \emph{blows-up} at $\bar{t} \in \R \cup\{\pm\infty\}$ if there exists a $w \in \R$ such that 
\begin{equation}\label{eq:blowup}
\lim_{t \to \bar{t}} |w^* V(t) w| \to +\infty.
\end{equation}
If for all $w$ such that~\eqref{eq:blowup} holds we have that $\lim_{t \to \bar{t}} w^* V(t) w = +\infty$ (resp $-\infty$), we say that $V(t)$ blows-up to $+\infty$ (resp. $-\infty$). 
The problem~\eqref{eq:wellpos} is related with a Hamiltonian system, similar to Jacobi equation~\eqref{eq:Jacobimatrix}.
\begin{lemma}[Relation with Jacobi]\label{l:relationjac}
Let $M(t),N(t)$ be the solution of the Jacobi equation
\begin{equation}\label{eq:hamsys}
\frac{d}{dt}\begin{pmatrix}
 M \\
 N
\end{pmatrix} = \begin{pmatrix}
-A^* & -Q(t) \\
B & A
\end{pmatrix}\begin{pmatrix}
M \\
N
\end{pmatrix}, \qquad M(0) = \mathbbold{1}, \quad N(0) = \mathbbold{0}.
\end{equation}
Then $N(t)$ is invertible for small $t>0$. Let $\bar{t}$ the first positive time such that $\det N(t) = 0$, and let $V(t)$ be the solution of
\begin{equation}\label{eq:ricsys}
\dot{V} +  A^* V + V A + Q(t) + V B V = 0, \qquad \lim_{t \to 0^+} V^{-1} =0,
\end{equation}
defined on its maximal interval $I$. Then $V(t)=M(t)N(t)^{-1}$ and $I= (0,\bar{t})$.
\end{lemma}
\begin{proof}
Let $V(t)$ be the solution of~\eqref{eq:ricsys} on $I=(0,a)$. We first show that it must be of the form $M(t)N(t)^{-1}$ on $(0,\bar{t})$ and then we prove that $\bar{t} = a$. By Lemma~\ref{l:wellpos}, $W(t):= V(t)^{-1}$ is well defined on $(0,\epsilon)$ and $\lim_{t \to 0^+} W(t) = 0 =:W(0)$. Consider then the solution $\tilde{M}(t)$ of
\begin{equation}
\dot{\tilde{M}} = -(A^* + Q(t) W(t))\tilde{M}, \qquad \tilde{M}(0) = \mathbbold{1},
\end{equation}
well defined at least on $[0,\epsilon)$. Then set $\tilde{N}(t):= W(t) \tilde{M}(t)$. Again by Lemma~\ref{l:wellpos}, $W(t) >0 $ on $(0,\epsilon)$, hence $\tilde{N}(t)$ is invertible for $t$ sufficiently small and $V(t) = \tilde{M}(t) \tilde{N}(t)^{-1}$ for small $t$. One can check that $\tilde{M}(t),\tilde{N}(t)$ solve~\eqref{eq:hamsys}, with the correct initial condition, hence $\tilde{M}(t),\tilde{N}(t) = M(t),N(t)$ on $[0,\epsilon)$. Then for all $t \in (0,\bar{t})$, the matrix $M(t)N(t)^{-1}$ is well defined and coincides with the solution $V(t)$ of~\eqref{eq:ricsys} on the interval $(0,\bar{t})$. In particular $a \geq \bar{t}$.

By contradiction, assume $a>\bar{t}$. Consider the two $n$-dimensional families of subspaces $L_1(t)$ and $L_2(t)$ of $\R^{2n}$ generated by the columns of
\begin{equation}
L_1(t):=\spn\begin{pmatrix}
M(t) \\
N(t)
\end{pmatrix}, \qquad \text{and} \qquad L_2(t):=\begin{pmatrix}
V(t) \\
\mathbbold{1}
\end{pmatrix},
\end{equation}
respectively. These may be seen as two curves in the Grassmannian of $n$-planes of $\R^{2n}$, both defined at least on $I=(0,a)$. We show that $L_1(t) = L_2(t)$ on $(0,\bar{t})$: indeed if $z_1,\ldots,z_n$ are the columns generating $L_1$ and $z_1',\ldots,z_n'$ are the columns generating $L_2$, then $z_i'= \sum_{j} N^{*-1}_{ij}(t)z_j$. By continuity, $L_1(\bar{t}) = L_2(\bar{t})$. This is absurd, since if $x \in \ker N(\bar{t}) \neq \{0\}$, then the vector $(0,x)^*$ is orthogonal to $L_1(\bar{t})$ but not to $L_2(\bar{t})$.
\end{proof}
\begin{corollary}[Relation with first conjugate time]\label{c:relationwithfirst}
Let $V(t)$ be the solution of the Riccati Cauchy problem~\eqref{eq:riccati} associated with the Jacobi equation along $\lambda(t)$. Then the maximal interval of definition is $I=(0,t_*)$, where $t_*$ is the first conjugate time along the geodesic.
\end{corollary}
The next theorem is a special version of~\cite[Theorem 40]{BR-comparison} for our setting.
\begin{theorem}[Riccati comparison theorem]\label{t:riccaticomparison}
Let $A,B$ be two $n\times n$ matrices satisfying the Kalman condition~\eqref{eq:Kalman}. Let $Q_1(t)$ and $Q_2(t)$ be smooth families of $n\times n$ symmetric matrices. Let $V_1(t)$ and $V_2(t)$ be the solutions of the Riccati Cauchy problems with limit initial data:
\begin{equation}
\dot{V}_i + A^* V_i + V_i A + Q_i(t) + V_i B V_i  = 0, \qquad \lim_{t \to 0^+} V_i^{-1} = 0,
\end{equation}
for $i=1,2$, defined on a common interval $I =(0,a)$. If $Q_1(t) \geq Q_2(t)$ for all $t \in I$, then $V_1(t) \leq V_2(t)$ for all $t \in I$.
\end{theorem}
A crucial property for comparison is the following \cite[Lemma 27]{BR-comparison}.
\begin{lemma}
Let $V(t)$ be a solution of the Cauchy problem~\eqref{eq:wellpos}.
If $0<\bar{t}<+\infty$ is a blow-up time for $V(t)$, then the latter blows up to $-\infty$.
\end{lemma}
\begin{corollary}\label{c:riccaticomparison}
Under the hypotheses of Theorem~\ref{t:riccaticomparison}, let $0<\bar{t}_i\leq +\infty$ be the blow-up time of $V_i$, for $i=1,2$. Then $\bar{t}_1 \leq \bar{t}_2$.
\end{corollary}
The typical scenario is a bound $Q_1(t) \geq Q_2$ with a constant symmetric matrix. To have a meaningful estimate, it is desirable that $\bar{t}_2 < +\infty$. We reformulate the results of~\cite{ARS-LQ} to give necessary and sufficient conditions for finite blow-up time of Riccati equations with \emph{constant} coefficients.
\begin{theorem}[Finiteness of blow-up times {\cite[Theorem A]{ARS-LQ}}]\label{t:silveira}
The solution of the Riccati Cauchy problem
\begin{equation}
\dot{V} + A^* V + V A + Q + VB V = 0, \qquad \lim_{t \to 0^+} V^{-1} = 0,
\end{equation}
has a finite blow-up time $\bar{t}(A,B,Q)$ if and only if the associated Hamiltonian matrix
\begin{equation}
\begin{pmatrix}
-A^* & -Q \\
B & A
\end{pmatrix}
\end{equation}
has at least one Jordan block of odd dimension, associated with a purely imaginary eigenvalue.
\end{theorem}
If one is able to compute the sub-Riemannian curvature matrix $R(t)$ of~\eqref{eq:Jacobimatrix}, and bound it with a (possibly constant) symmetric matrix $\bar{R}$, then one can apply the comparison theory described so far to estimate the first conjugate time $t_*(\gamma)$ along the sub-Riemannian geodesic with the first blow-up time $t(\bar{R},A,B)$ of the Riccati equation associated with the matrices $A$, $B$ and $Q(t) = \bar{R}$. Theorem~\ref{t:silveira} then provides conditions on $\bar{R}$ such that $t(\bar{R},A,B) < +\infty$.

The advantage of this formulation (in terms of blow-up times for the Riccati equation) is that the latter can be suitably ``traced'', to obtain comparison theorems with weaker assumptions on the average curvature (Ricci-type curvature) instead of the full sectional-type curvature $R(t)$. In the Riemannian case (i.e. when $A = \mathbbold{0}$, $B = \mathbbold{1}$), this is well known. As we show, in the sub-Riemannian case the tracing procedure is much more delicate.

\section{Proof of the results} \label{s:proofs}
Let now $V(t)$ be the solution of the Riccati Cauchy problem~\eqref{eq:riccati} associated with the Jacobi equation~\eqref{eq:Jacobimatrix} along a given extremal $\lambda(t)$. For convenience, we recall that $V(t)$ solves
\begin{equation}\label{eq:riccati2}
\dot{V} + A^* V + V A +R(t)+ V B V =0, \qquad \lim_{t \to 0^+} V^{-1} = 0,
\end{equation}
\begin{equation}
A = \begin{pmatrix}
\mathbbold{0} & \mathbbold{1} & \mathbbold{0} \\ \mathbbold{0} & \mathbbold{0} & \mathbbold{0} \\ \mathbbold{0} & \mathbbold{0} & \mathbbold{0}
\end{pmatrix}, \quad B = \begin{pmatrix}
\mathbbold{0} & \mathbbold{0} & \mathbbold{0} \\ \mathbbold{0} & \mathbbold{1} & \mathbbold{0} \\ \mathbbold{0} & \mathbbold{0} &\mathbbold{1}
\end{pmatrix}, \quad
R(t) = \begin{pmatrix}
R_{aa}(t) & R_{ab}(t) & R_{ac}(t) \\
R_{ba}(t) & R_{bb}(t) & R_{bc}(t) \\
R_{ca}(t) & R_{cb}(t) & R_{cc}(t)
\end{pmatrix},
\end{equation}
with $R(t)$ symmetric and $R_{ab}(t) = - R_{ab}(t)^*$. In the notation of Section~\ref{sec:sasakiannormalframe}, we decompose 
\begin{equation}
V(t)=\begin{pmatrix}
V_{aa}(t) & V_{ab}(t) & V_{ac}(t) \\
V_{ba}(t) & V_{bb}(t) & V_{bc}(t) \\
V_{ca}(t) & V_{cb}(t) & V_{cc}(t) \\
\end{pmatrix},
\end{equation}
where $V_{\alpha\beta}$ is a $|\alpha|\times |\beta|$ matrix, $\alpha,\beta=a,b,c$. Notice the special structure of $A$ and $B$:
\begin{equation}\label{eq:splitI-IIa}
A = \begin{pmatrix}
A_{\mathrm{I}} & \mathbbold{0} \\
\mathbbold{0} & A_{\mathrm{II}} 
\end{pmatrix}, \qquad B= \begin{pmatrix}
B_{\mathrm{I}} & \mathbbold{0} \\
\mathbbold{0} & B_{\mathrm{II}}
\end{pmatrix},
\end{equation}
where $A_{\mathrm{I}},B_{\mathrm{I}}$ are $(2n-2k)\times (2n-2k)$ blocks and $A_{\mathrm{II}},B_{\mathrm{II}}$ are $(2k-n )\times (2k-n)$ blocks:
\begin{equation}\label{eq:splitI-IIb}
A_{\mathrm{I}} := \begin{pmatrix}
\mathbbold{0} & \mathbbold{1} \\
\mathbbold{0} & \mathbbold{0}
\end{pmatrix}, \qquad B_{\mathrm{I}} := \begin{pmatrix}
\mathbbold{0} & \mathbbold{0} \\
\mathbbold{0} & \mathbbold{1}
\end{pmatrix}, \qquad A_{\mathrm{II}} := \mathbbold{0}, \qquad B_{\mathrm{II}} := \mathbbold{1}.
\end{equation}
Analogously, we consider the two symmetric matrices (recall that $V(t)$ itself is symmetric)
\begin{equation}\label{eq:splitI-IIc}
V_{\mathrm{I}}(t) := \begin{pmatrix}
V_{aa}(t) & V_{ab}(t) \\
V_{ba}(t) & V_{bb}(t)
\end{pmatrix}, \qquad V_{\mathrm{II}}(t) := V_{cc}(t),
\end{equation}
which are $(2n-2k)\times (2n-2k)$ and $(2k-n)\times (2k-n)$ diagonal blocks of $V(t)$, respectively.
\begin{lemma}\label{l:explosion}
The families $V_{\mathrm{I}}(t)$ and $V_{\mathrm{II}}(t)$ are invertible for small $t>0$ and 
\begin{equation}
\lim_{t\to 0^+}(V_{\mathrm{I}})^{-1}=0, \qquad \lim_{t\to 0^+}(V_{\mathrm{II}})^{-1} = 0.
\end{equation}
\end{lemma}
\begin{proof}
We prove it for $V_{\mathrm{I}}(t)$. Suppressing the explicit dependence on $t$, we have
\begin{equation}\label{eq:sblock}
V = \begin{pmatrix}
V_{\mathrm{I}} & V_\mathrm{III} \\
V_\mathrm{III}^* & V_{\mathrm{II}}
\end{pmatrix}, \qquad \text{with} \qquad V_{\mathrm{III}} = \begin{pmatrix} V_{ac} \\ V_{bc} \end{pmatrix}.
\end{equation}
We partition similarly the inverse matrix $W := V^{-1}$. In particular, by block-wise inversion, $W_{\mathrm{I}} = (V^{-1})_{\mathrm{I}} = (V_{\mathrm{I}} - V_{\mathrm{III}} (V_{\mathrm{II}})^{-1}V_{\mathrm{III}}^*)^{-1}$. By Lemma~\ref{l:wellpos}, $V > 0$ for small $t>0$, in particular $V_{\mathrm{II}}>0$  on the same interval. Moreover, also $W >0$ and then $W_{\mathrm{I}}>0$.

Then $V_{\mathrm{I}} - (W_{\mathrm{I}})^{-1} = V_{\mathrm{III}} (V_{\mathrm{II}})^{-1}V_{\mathrm{III}}^* \geq 0$. Thus $V_{\mathrm{I}} \geq (W_{\mathrm{I}})^{-1}>0$. Taking the inverse, by positivity, $0 < (V_{\mathrm{I}})^{-1} \leq W_{\mathrm{I}}$ for small $t>0$. Since $\lim_{t\to 0^+}W_{\mathrm{I}} = \lim_{t \to 0^+} (V^{-1})_{\mathrm{I}} = 0$, we obtain the result. Similarly for $V_{\mathrm{II}}$.
\end{proof}

\subsection{Proof of Theorem~\ref{t:conjII-intro}}

The first conjugate time $t_*(\gamma)$ is the first blow-up time of $V(t)$, solution of~\eqref{eq:riccati2}. Using also Lemma~\ref{l:explosion}, we see that the $(2k-n) \times (2k-n)$ block $V_{\mathrm{II}}$ solves
\begin{equation}
\dot{V}_{\mathrm{II}} + R_{\mathrm{II}}(t) + V_{\mathrm{II}}^2 =0 , \qquad \lim_{t \to 0^+} (V_{\mathrm{II}})^{-1} = 0,
\end{equation}
with $R_{\mathrm{II}}(t) = R_{cc}(t) + V_{cb}(V_{cb})^* \geq R_{cc}(t)$. First, we ``take out the direction of motion'' (this procedure is the classical Riemannian one, see \cite[Chapter 14]{villanibook}). According to Remark~\ref{rmk:dimensionalreduction}, we can assume $R_{cc}(t)$ has the following block structure
\begin{equation}
R_{cc}(t) = \begin{pmatrix}
R_{cc}^\prime(t) & 0 \\
0 & 0
\end{pmatrix},
\end{equation}
where $R_{cc}(t)$ has dimension $2k-n-1$. Accordingly, the solution $V_{\mathrm{II}}$ has the form
\begin{equation}\label{eq:splitII-II'}
V_{\mathrm{II}} = \begin{pmatrix}
V_{\mathrm{II}}^\prime & 0 \\ 
0  & v^0_{\mathrm{II}}  
\end{pmatrix},
\end{equation}
where $V^\prime_{\mathrm{II}}$ is a $(2k-n-1)\times(2k-n-1)$ matrix and $v^0_{\mathrm{II}}$ is a $1\times 1 $ matrix. They satisfy
\begin{align}
\dot{V}^\prime_{\mathrm{II}} + R^\prime_{\mathrm{II}}(t) + (V^\prime_{\mathrm{II}})^2=0, \qquad \lim_{t \to 0^+} (V^\prime_{\mathrm{II}})^{-1} =0, \qquad & \text{with}\quad R^\prime_{\mathrm{II}}(t) \geq R^\prime_{cc}(t), \label{eq:viiorth}\\
\dot{v}^0_{\mathrm{II}} + r^0_{\mathrm{II}}(t) + (v^0_{\mathrm{II}})^2=0, \qquad \lim_{t \to 0^+} (v_{\mathrm{II}}^0)^{-1} =0, \qquad & \text{with} \quad  r^0_{\mathrm{II}}(t)  \geq 0 \label{eq:viipar}.
\end{align}
By Theorem~\ref{t:riccaticomparison}, $v_\mathrm{II}^0$ is controlled by the solution of~\eqref{eq:viipar} with $r^0_{\mathrm{II}}(t)\equiv 0$, that is $v_\mathrm{II}^0(t) \leq 1/t$. This term gives no contribution to conjugate time (indeed $1/t$ has no finite blow-up time for $t>0$) but we will use $v_\mathrm{II}^0(t) \leq 1/t$ in a subsequent proof hence it was worth pointing it out.
Now we turn to~\eqref{eq:viiorth}. Its normalized trace 
\begin{equation}\label{eq:smallv'def}
v^\prime_{\mathrm{II}}(t):=\frac{1}{2k-n-1} \trace (V^\prime_{\mathrm{II}}(t))
\end{equation}
solves
\begin{equation}
\dot{v}^\prime_{\mathrm{II}} + r^\prime_{\mathrm{II}}(t) + (v^\prime_{\mathrm{II}})^2 = 0, \qquad \lim_{t\to 0^+} (v^\prime_{\mathrm{II}})^{-1} = 0,
\end{equation}
with (suppressing the explicit dependence on $t$)
\begin{equation}
r^\prime_{\mathrm{II}}  = \frac{\trace (R^\prime_{\mathrm{II}})}{2k-n-1} + \frac{\left[(2k-n-1)\trace((V^\prime_{\mathrm{II}})^2) -\trace(V^\prime_{\mathrm{II}})^2 \right]}{(2k-n-1)^2} \geq \frac{\mathfrak{Ric}^c}{2k-n-1} \geq \kappa_c,
\end{equation}
where we used that, for an $m \times m$ symmetric matrix $M$, $\trace(M^2) \geq \tfrac{1}{m}\trace(M)^2$ and $\trace( R^\prime_{\mathrm{II}}) \geq \trace (R^\prime_{cc}) = \Riccan^c$.  Then, applying Theorem~\ref{t:riccaticomparison}, $v_{\mathrm{II}}'(t) \leq v_{\kappa_c}(t)$, where $v_{\kappa_c}$ is the solution of
\begin{equation}\label{eq:riccaticomparisonII}
\dot{v}_{\kappa_c} + \kappa_c + v_{\kappa_c}^2 = 0, \qquad \lim_{t\to 0^+} v_{\kappa_c}^{-1} = 0.
\end{equation}
In particular, $t_*(\gamma) \leq \bar{t}(\kappa_c)$, where $\bar{t}(\kappa_c)$ is the first blow-up time of $v_{\kappa_c}$. In this case, we can compute the explicit solution of~\eqref{eq:riccaticomparisonII}, which is
\begin{equation}
v_{\kappa_c}(t) = \begin{cases}
\sqrt{\kappa_c} \cot(\sqrt{\kappa_c} t) & \kappa_c >0, \\
\frac{1}{t}  & \kappa_c = 0, \\
\sqrt{|\kappa_c|} \coth(\sqrt{|\kappa_c|} t) & \kappa_c <0. \end{cases}
\end{equation}
Thus, when $\kappa_c >0$, we have $t_*(\gamma) \leq \bar{t}(\kappa_c) = \pi/\sqrt{\kappa_c}$. \hfill $\qed$
\begin{rmk}\label{rmk:defskc}
For later use, we rename $s_{\kappa_c}(t) := v_{\kappa_c}(t)$ and we observe that
\begin{equation}
s_{\alpha^2 \kappa_c}(t) = \alpha s_{\kappa_c}(\alpha t), \qquad \forall \alpha >0,
\end{equation}
for  all $t>0$ where it makes sense.
\end{rmk}

\subsection{Proof of Theorem~\ref{t:conjI-intro}}

The first conjugate time $t_*(\gamma)$ is the first blow-up time to $V(t)$, solution of~\eqref{eq:riccati2}. The $(2n-2k) \times (2n-2k)$ block $V_{\mathrm{I}}$ solves
\begin{equation}
\dot{V}_{\mathrm{I}} + A^*_{\mathrm{I}}V_{\mathrm{I}} + V_{\mathrm{I}} A_{\mathrm{I}} + R_{\mathrm{I}}(t) + V_{\mathrm{I}} B_\mathrm{I} V_{\mathrm{I}} =0 , \qquad \lim_{t \to 0^+} V_{\mathrm{I}}^{-1} = 0,
\end{equation}
where
\begin{equation*}
A_{\mathrm{I}} = \begin{pmatrix}
\mathbbold{0} & \mathbbold{1} \\
\mathbbold{0} & \mathbbold{0}
\end{pmatrix}, \qquad B_{\mathrm{I}} = \begin{pmatrix}
\mathbbold{0} & \mathbbold{0} \\
\mathbbold{0} & \mathbbold{1}
\end{pmatrix},\qquad
R_{\mathrm{I}}(t) = \begin{pmatrix}
R_{aa}(t) & R_{ab}(t) \\
R_{ba}(t) & R_{bb}(t)
\end{pmatrix}+ \begin{pmatrix}
V_{ac} \\ V_{bc}
\end{pmatrix}\begin{pmatrix}
V_{ac}^* & V_{bc}^*
\end{pmatrix}
\end{equation*}
Taking the normalized block-wise trace, that is
\begin{equation}\label{eq:smallvdef}
v_{\mathrm{I}}(t):=\frac{1}{n-k}\begin{pmatrix}
\trace(V_{aa}(t)) & \trace(V_{ab}(t)) \\
\trace(V_{ba}(t)) & \trace(V_{bb}(t))
\end{pmatrix},
\end{equation}
we observe that $v_{\mathrm{I}}$ solves the following $2 \times 2$ Riccati Cauchy problem
\begin{equation}\label{eq:riccaticomparisonI}
\dot{v}_{\mathrm{I}} + a_{\mathrm{I}}^* v_{\mathrm{I}} + v_{\mathrm{I}} a_{\mathrm{I}} + r_{\mathrm{I}}(t) + v_{\mathrm{I}} b_{\mathrm{I}} v_{\mathrm{I}} = 0, \qquad \lim_{t \to 0^+}  v_{\mathrm{I}}^{-1} = 0,
\end{equation}
with
\begin{equation}\label{eq:smallabdef}
 \qquad a_{\mathrm{I}}:= \begin{pmatrix}
0 & 1 \\
0 & 0
\end{pmatrix}, \qquad b_{\mathrm{I}}:= \begin{pmatrix}
0 & 0 \\
0 & 1
\end{pmatrix}
\end{equation}
and, suppressing the explicit dependence on $t$,
\begin{multline}
r_{\mathrm{I}}(t) := \frac{1}{n-k}\begin{pmatrix}
\trace(R_{aa}) & \trace(R_{ab}) \\
\trace(R_{ba}) & \trace(R_{bb})
\end{pmatrix} + \frac{1}{n-k}\begin{pmatrix}
\trace(V_{ac}V_{ac}^*) & \trace(V_{ac} V_{bc}^*) \\ \trace(V_{bc} V_{ac}^*) & \trace(V_{bc} V_{bc}^*)
\end{pmatrix}+\\ 
+ \frac{1}{n-k}\left[\begin{pmatrix}
\trace(V_{ab} V_{ab}^*) & \trace(V_{ab} V_{bb}) \\
\trace(V_{bb} V_{ab}^*) & \trace(V_{bb}V_{bb})
\end{pmatrix} -\frac{1}{n-k} \begin{pmatrix}
\trace(V_{ab})\trace( V_{ab}^*) & \trace(V_{ab})\trace(  V_{bb}) \\
\trace(V_{bb})\trace(  V_{ab}^*) & \trace(V_{bb})\trace( V_{bb})
\end{pmatrix}   \right].
\end{multline}
The second term is non-negative. In fact the minors $\trace(V_{ac} V_{ac}^*)$, $\trace(V_{bc} V_{bc}^*)$ and the determinant $\trace(V_{ac} V_{ac}^*)\trace(V_{bc}V_{bc}^*) - \trace(V_{ac}V_{bc}^*)^2 \geq 0$ are non-negative, by the Cauchy-Schwarz inequality. Also the last term is non-negative
\begin{equation}\label{eq:ineq-traceV}
\begin{pmatrix}
\trace(V_{ab} V_{ab}^*) & \trace(V_{ab} V_{bb}) \\
\trace(V_{bb} V_{ab}^*) & \trace(V_{bb}V_{bb})
\end{pmatrix} -\frac{1}{n-k} \begin{pmatrix}
\trace(V_{ab})\trace( V_{ab}^*) & \trace(V_{ab})\trace(  V_{bb}) \\
\trace(V_{bb})\trace(  V_{ab}^*) & \trace(V_{bb})\trace( V_{bb})
\end{pmatrix} \geq 0.
\end{equation}
To prove \eqref{eq:ineq-traceV} it is enough to show that the principal determinants are non-negative, i.e.
\begin{equation}\label{eq:principal-det1}
\trace(V_{ab}V_{ab}^*) - \frac{\trace(V_{ab})\trace( V_{ab}^*)}{n-k}\geq 0, \qquad \trace(V_{bb}V_{bb}^*) - \frac{\trace(V_{bb})\trace( V_{bb}^*)}{n-k}\ge 0,
\end{equation}
(that follow from the Cauchy-Schwarz inequality) and the determinant is non-negative:
\begin{multline}\label{eq:principal-det2}
\trace(V_{ab}V_{ab}^*)\trace(V_{bb}V_{bb}^*) - \trace(V_{ab} V_{bb}^*)^2 - \frac{\trace(V_{bb})^2\trace(V_{ab}V_{ab}^*)}{n-k} - \frac{\trace(V_{ab})^2\trace(V_{bb}V_{bb}^*)}{n-k} \\
+ \frac{2 \trace(V_{ab})\trace(V_{bb})\trace(V_{ab}V_{bb}^*)}{n-k}\ge 0.
\end{multline}
Inequality~\eqref{eq:principal-det2} follows from the next lemma (with $X= V_{ab}$, $Y=V_{bb}$ and $m=n-k$).
\begin{lemma}
	Let $M_m(\mathbb{R})$ be the real vector space of real $m\times m$ matrices with scalar product $\langle X,Y\rangle:=\trace(XY^*)$. Then the following inequality holds true for all $X,Y \in M_m(\R)$
	\begin{equation}\label{eq:ineq-trace}
	\|X\|^2\|Y\|^2-\langle X,Y\rangle^2 + \frac{2}{m}\trace(X)\trace(Y)\langle X,Y\rangle \ge \frac{1}{m}\left(\trace(Y)^2\|X\|^2 + \trace(X)^2\|Y\|^2\right).
	\end{equation}
\end{lemma}	

\begin{proof}
	If $\|X\|=0$ the statement is trivially true. Suppose $\|X\|>0$ and write $Z = Y- \frac{\langle X,Y\rangle}{\|X\|^2}X$. One can check that \eqref{eq:ineq-trace} is equivalent to
	\begin{equation}\label{eq:ineq-trace2}
	\|X\|^2\|Z\|^2 \ge \frac{1}{m}\left(\trace(Z)^2\|X\|^2 + \trace(X)^2\|Z\|^2\right).
	\end{equation}
		If $\trace(X)=0$ then \eqref{eq:ineq-trace2} follows from $\|X\|^2\ge \frac{1}{m}\trace(X)^2$. Suppose $\trace(X),\trace(Z)\ne 0$, hence \eqref{eq:ineq-trace2} is equivalent to
	\begin{equation}\label{eq:ineq-trace3}
	\|X\|^2\|Z\|^2 \ge \frac{1}{m}\left(\|X\|^2 + \|Z\|^2\right),
	\end{equation}
	where $\trace(X)=\trace(Z)=1$ and $\langle X,Z \rangle=0$. 
Define the matrix
\begin{equation}
W := \frac{Z \|X\|^2 + X \|Z\|^2}{\|X\|^2+ \|Z\|^2}.
\end{equation}
By Cauchy-Schwarz inequality $m \|W\|^2 \geq \trace(W)^2 =1 $, and this corresponds to~\eqref{eq:ineq-trace3}.
\end{proof}

Finally, by~\eqref{eq:normalcond}, $R_{ab}(t)$ is skew-symmetric, thus (suppressing explicit dependence on $t$)
\begin{equation}
r_{\mathrm{I}}(t) \geq \frac{1}{n-k}\begin{pmatrix}
\trace(R_{aa}) & \trace(R_{ab}) \\
\trace(R_{ba}) & \trace(R_{cc})
\end{pmatrix} = \frac{1}{n-k}\begin{pmatrix}\Riccan^a & 0 \\
0 & \Riccan^b
\end{pmatrix} \geq \begin{pmatrix}
\kappa_a & 0 \\
0 & \kappa_b
\end{pmatrix}.
\end{equation}
By Theorem~\ref{t:riccaticomparison}, $v_{\mathrm{I}}(t) \leq v_{\kappa_a,\kappa_b}(t)$, where $v_{\kappa_a,\kappa_b}(t)$ is the solution of~\eqref{eq:riccaticomparisonI} with $r_{\mathrm{I}}(t)$ replaced by the constant $2\times 2$ matrix $q_{\mathrm{I}} = \mathrm{diag}(\kappa_a,\kappa_b)$. The blow-up time of $v_{\kappa_a,\kappa_b}(t)$ is $\bar{t}(\kappa_a,\kappa_b)$. A blow-up of $v_{\mathrm{I}}$ implies a blow-up of $V_{\mathrm{I}}$ and $V$. Then, $t_*(\gamma) \leq \bar{t}(\kappa_a,\kappa_b)$. The next proposition characterizes $\bar{t}(\kappa_a,\kappa_b)$ and, in particular, it shows that under conditions~\eqref{eq:conditions}, $\bar{t}(\kappa_a,\kappa_b)$ is finite (this proves also Proposition~\ref{p:stima-intro}). \hfill $\qed$

\begin{proposition}\label{p:stima}
Consider the following Cauchy problem with a $2 \times  2$ matrix Riccati equation
\begin{equation}\label{eq:RiccatiTypeI}
\dot{v}_{\kappa_a,\kappa_b} + a^*_{\mathrm{I}} v_{\kappa_a,\kappa_b} + v_{\kappa_a,\kappa_b} a_{\mathrm{I}} + q_{\mathrm{I}} + v_{\kappa_a,\kappa_b} b_{\mathrm{I}} v_{\kappa_a,\kappa_b} = 0, \qquad \lim_{t \to 0^+} v_{\kappa_a,\kappa_b}^{-1} = 0,
\end{equation}
with constant matrix coefficients
\begin{equation}
a_{\mathrm{I}} = \begin{pmatrix}
0 & 1 \\
0 & 0
\end{pmatrix}, \qquad b_{\mathrm{I}} = \begin{pmatrix}
0 & 0 \\
0 & 1
\end{pmatrix}, \qquad q_{\mathrm{I}}= \begin{pmatrix}
\kappa_a & 0 \\
0 & \kappa_b
\end{pmatrix}, \qquad \kappa_a,\kappa_b \in \R.
\end{equation}
The first blow-up time $\bar{t}(\kappa_a,\kappa_b)$ of the solution of~\eqref{eq:RiccatiTypeI} is the first blow-up time of the function $s_{\kappa_a,\kappa_b}: (0,\bar{t}(\kappa_a,\kappa_b)) \to \R$, given by
\begin{equation}
s_{\kappa_a,\kappa_b}(t):= \frac{2}{t}\left(\frac{\sinc \left(2 \theta_+ t\right)-\sinc \left(2 \theta _- t\right)}{\sinc\left(\theta _+ t\right)^2- \sinc\left(\theta _- t\right)^2}\right), \qquad \theta_{\pm}= \frac{1}{2}(\sqrt{x+y} \pm \sqrt{x-y}), 
\end{equation}
where $\sinc(a) = \sin(a)/a$ and we set $x = \frac{\kappa_b}{2}$ and $y = \frac{\sqrt{4\kappa_a + \kappa_b^2}}{2}$.
Moreover 
\begin{equation}\label{eq:stima}
\bar{t}(\kappa_a,\kappa_b) \leq \frac{\pi}{\mathrm{Re}(\theta_-)} = \frac{2 \pi}{\mathrm{Re}(\sqrt{x+y}-\sqrt{x-y})},
\end{equation}
where the r.h.s. of~\eqref{eq:stima} is $+\infty$ if the denominator is zero and $\sqrt{\cdot}$ is the principal value of the square root. The equality holds if and only if $\kappa_a = 0$, in this case $\bar{t}(0,\kappa_b) = 2\pi/\sqrt{\kappa}_b$. In particular $\bar{t}(\kappa_a,\kappa_b)$ is finite if and only if
\begin{equation}
\begin{cases} \kappa_b \geq 0, & \\
\kappa_b^2 + 4\kappa_a > 0, &
\end{cases} \qquad\text{or} \qquad \begin{cases} \kappa_b < 0, & \\
\kappa_a > 0. &
\end{cases}
\end{equation}
\end{proposition}
\begin{proof}
To compute $\bar{t}(\kappa_a,\kappa_b)$ we use Lemma~\ref{l:relationjac}. Then $v_{\kappa_a,\kappa_b}(t) = m(t)n(t)^{-1}$ with
\begin{equation}
\begin{pmatrix}
m(t) \\
n(t)
\end{pmatrix} = \exp\left(t \begin{pmatrix}
-a_{\mathrm{I}}^* & -q_{\mathrm{I}} \\
b_{\mathrm{I}} & a_{\mathrm{I}}
\end{pmatrix}\right) \begin{pmatrix}
\mathbbold{1} \\ 
\mathbbold{0}
\end{pmatrix},
\end{equation}
where $\exp$ is the matrix exponential. Thus $\bar{t}(\kappa_a,\kappa_b)$ is the first positive zero of $\det n(t)$ or. For reasons that will be clear later, it is more convenient to study, equivalently, the first blow-up time of
\begin{equation}
s_{\kappa_a,\kappa_b}(t):= \frac{d}{d t}\log |\det n(t)|, \qquad t \in (0,\bar{t}(\kappa_a,\kappa_b)).
\end{equation}
\begin{rmk}\label{rmk:defskakb}
For later use, observe that
\begin{equation}
s_{\kappa_a,\kappa_b}(t) = \trace(\dot n(t) n(t)^{-1}) = \trace(b_{\mathrm{I}}(t) v_{\kappa_a,\kappa_b}(t)),
\end{equation}
and the function $s_{\kappa_a,\kappa_b}(t)$ has the following homogeneity property:
\begin{equation}
s_{\alpha^4 \kappa_a,\alpha^2 \kappa_b}(t) = \alpha s_{\kappa_a,\kappa_b}(\alpha t), \qquad \forall \alpha >0,
\end{equation}
for  all $t>0$ where it makes sense.
\end{rmk}
We compute $s_{\kappa_a,\kappa_b}(t)$. The characteristic polynomial of
\begin{equation}\label{eq:hammat}
\begin{pmatrix}
-a_{\mathrm{I}}^* & -q_{\mathrm{I}} \\
b_{\mathrm{I}} & a_{\mathrm{I}}
\end{pmatrix} = \begin{pmatrix}
0 & 0 & -\kappa_a & 0 \\
-1 & 0 & 0 & -\kappa_b \\
0 & 0 & 0 & 1 \\
0 & 1 & 0 & 0
\end{pmatrix}
\end{equation}
is $P(x) = x^4 + \kappa_b x^2 -\kappa_a$. Recall that~\eqref{eq:hammat} is a Hamiltonian matrix, hence if $\lambda \in \mathbb{C}$ is an eigenvalue, then also $\pm \lambda$ and $\pm \bar\lambda$ are eigenvalues (the bar denotes complex conjugation). Its Jordan form depends on the value of $\Delta := \kappa_b^2 + 4\kappa_a$:
\begin{itemize}
\item[(i)] If $\Delta = 0$ there are two Jordan blocks (of size $2$) associated with eigenvalues $\pm \sqrt{-\kappa_b/2}$,
\item[(ii)] If $\Delta < 0$ then~\eqref{eq:hammat} has $4$ distinct simple eigenvalues $\pm \lambda, \pm \bar\lambda \in \mathbb{C}$,
\item[(iii)] If $\Delta > 0$ then~\eqref{eq:hammat} has $2$ pairs $\pm \lambda_1$ and $\pm \lambda_2$, with $\lambda_1 \neq \pm \lambda_2$ of simple eigenvalues.
\end{itemize}
In the cases (i) and (ii) $\bar{t}(\kappa_a,\kappa_b) = +\infty$ by Theorem~\ref{t:silveira}. In the remaining case, set:
\begin{equation}
x= \frac{\kappa_b}{2}, \qquad y=\frac{\sqrt{\Delta}}{2}, \qquad \theta_{\pm}=\frac{1}{2}(\sqrt{x+y}\pm \sqrt{x-y}).
\end{equation}
In particular we recover $\kappa_b =2(\theta_+^2+\theta_-^2)$ and $\kappa_a = -(\theta_+^2-\theta_-^2)^2$. The eigenvalues of~\eqref{eq:hammat} are given then by the two distinct pairs
\begin{equation}
\pm \lambda_1:=\pm i(\theta_++\theta_-), \qquad \pm\lambda_2:=\pm i(\theta_+-\theta_-).
\end{equation}
This encompasses different cases ($2$ distinct imaginary pairs, $2$ distinct real pairs, $1$ imaginary and $1$ real pair). The corresponding eigenvectors are
\begin{equation}
\xi_1^\pm = \begin{pmatrix}
-(\theta_--\theta_+)^2 \\ \pm i(\theta_-+\theta_+) \\ \frac{1}{\pm i(\theta_-+\theta_+)} \\ 1 
\end{pmatrix},  \qquad
\xi_2^\pm = \begin{pmatrix}
-(\theta_-+\theta_+)^2 \\ \pm i(\theta_+-\theta_-) \\ \frac{1}{\pm i(\theta_+-\theta_-)} \\ 1 
\end{pmatrix}.
\end{equation}
After some routine computations for the matrix exponential of~\eqref{eq:hammat} one obtains
\begin{equation}\label{eq:derlogdetN}
s_{\kappa_a,\kappa_b}(t)=\frac{2}{t}\left(\frac{\sinc \left(2 \theta_+ t\right)-\sinc \left(2 \theta _- t\right)}{\sinc\left(\theta _+ t\right)^2-\sinc\left(\theta _- t\right)^2}\right),
\end{equation}
where, if $\theta_+ = \pm \theta_-$, the result must be understood in the limit $\theta_+ \to \pm \theta_-$.
\paragraph{Case 1} The two pairs of eigenvalues are pure imaginary, that is $\theta_+ > \theta_- >0$ are reals. Then the first blow-up time of $s_{\kappa_a,\kappa_b}(t)$ is at the first positive root of
\begin{equation}\label{eq:sinc-sinc}
\sinc(\theta_+ t)^2 = \sinc(\theta_- t)^2.
\end{equation}
In particular, since $\theta_+ > \theta_- > 0$, and the first zero of $\sinc(a)$ is at $a = \pi$, we have
\begin{equation}
\frac{\pi}{\theta_+} < \bar{t}(\kappa_a,\kappa_b) < \frac{\pi}{\theta_-}.
\end{equation}

\paragraph{Case 2} The two pairs of eigenvalues are both real, that is $\theta_+,\theta_-$ are pure imaginary. We already know from Theorem~\ref{t:silveira} that in this case $\bar{t}(\kappa_a,\kappa_b) = + \infty$. We prove it directly.  If $|\theta_+| \neq |\theta_-|$, the first blow-up time of $s_{\kappa_a,\kappa_b}(t)$ is the first positive root of
\begin{equation}
\frac{\sinh(|\theta_+| t)^2}{|\theta_+|^2} = \frac{\sinh(|\theta_-| t)^2}{|\theta_-|^2},
\end{equation}
and since $|\theta_+|\neq |\theta_-|$ the above equation has no positive solutions. If $|\theta_+|= |\theta_-|$, then~\eqref{eq:derlogdetN} must be considered in the limit $\theta_+ \to \pm \theta_-$. After taking the limit, we obtain that the first blow-up time is the first positive root of $\tanh(|\theta_+| t) = |\theta_+| t$, that has no solution for $t \neq 0$.
\paragraph{Case 3} One pair is pure imaginary and the other is real. This means that $\theta_+ = \alpha+i \beta$ and $\theta_- = \alpha-i\beta$, with $\alpha > 0$ and $\beta \geq 0$. In this case~\eqref{eq:derlogdetN} becomes
\begin{equation}
s_{\kappa_a,\kappa_b}(t) = \frac{(\alpha^2+\beta^2)}{2 \cosh(\beta t)^2}\frac{\beta\sin(2\alpha t)\cosh(2\beta t )-\alpha\cos(2\alpha t)\sinh(2\beta t)}{[\beta\sin(\alpha t)-\alpha \cos(\alpha t)\tanh(\beta t)][\alpha\sin(\alpha t)+\beta\cos(\alpha t)\tanh(\beta t)]}.
\end{equation}
Assume first $\beta >0$. In this case, if $\cos(\alpha t) = 0$, then $s_{\kappa_a,\kappa_b}(t)$ is finite. Hence, assuming $\cos(\alpha t) \neq 0$, the first blow-up of $s_{\kappa_a,\kappa_b}(t)$ is given by the first positive root of
\begin{equation}
[\beta \tan(\alpha t)-\alpha \tanh(\beta t)][\alpha \tan(\alpha t)+ \beta \tanh(\beta t)] = 0.
\end{equation}
A rapid inspection shows that the first positive root occurs thanks to the second factor, and
\begin{equation}
\bar{t}(\kappa_a,\kappa_b) < \frac{\pi}{\alpha} = \frac{\pi}{\mathrm{Re}(\theta_-)}.
\end{equation}
The case $\beta =0$ corresponds to $\theta_+ = \theta_-$ and~\eqref{eq:derlogdetN} must be taken in the limit. We obtain
\begin{equation}
s_{\kappa_a,\kappa_b}(t) = \alpha  \left(\frac{\alpha  t}{1-\alpha  t \cot (\alpha  t)}+\cot (\alpha  t)\right),
\end{equation}
whose first blow-up time is $\bar{t}(\kappa_a,\kappa_b)= \frac{\pi}{\alpha} = \frac{\pi}{\mathrm{Re}(\theta_-)}$. This completes all the cases.
\end{proof}

\subsection{Proof of Theorems~\ref{t:bmI-intro} and~\ref{t:bmII-intro}}

By Theorem~\ref{t:conjI-intro} (or~\ref{t:conjII-intro}), under conditions~\ref{eq:conditions}, any length-parametrized sub-Rieman\-nian geodesic $\gamma$ has a conjugate time $t_*(\gamma) \leq \bar{t}(\kappa_a,\kappa_b)$ (resp. $\leq \bar{t}(\kappa_c)$). In particular, no geodesic can be optimal after such a length.

The sub-Riemannian structure is complete, hence for any pair $q,p \in M$ there exists a (possibly not-unique) minimizing trajectory joining $q$ and $p$ (see \cite{nostrolibro,montgomerybook,rifford2014sub}). This trajectory is a geodesic $\gamma_{p,q}$ (the structure is fat and there are no abnormal minimizers).
\begin{equation*}
\diam( M) = \inf\{\d(p,q) \mid p,q \in M\} = \inf\{\ell(\gamma_{p,q}) \mid p,q \in M \} \leq \bar{t}(\kappa_a,\kappa_b)\quad (\text{resp. } \bar{t}(\kappa_c)).
\end{equation*}
By completeness, closed balls are compact, hence $M$ is compact. The argument for the fundamental group is the classical one, considering the universal cover $\tilde{M}$ (see \cite{Myers}). \hfill $\qed$

\subsection{Proof of Theorem~\ref{t:lapl-intro}}

Fix $q_0 \in M$. The function $\f_{q_0}:=\frac{1}{2}\d(q_0,\cdot)^2$ on a complete, fat sub-Riemannian structure has the following properties (see \cite{nostrolibro,rifford2014sub}):
\begin{itemize}
\item is smooth on a maximal open dense set $\Sigma_{q_0}$, whose complement has zero measure;
\item for any point $q \in \Sigma_{q_0}$, there exists a unique minimizing geodesic $\gamma :[0,1] \to M$ such that $\gamma(0) = q_0$ and $\gamma(1)=q$. The corresponding final covector is given by
\begin{equation}
\lambda(1) = d_q \f_{q_0} \in T_{q}^*M,
\end{equation}
\end{itemize}
Notice that the initial covector $\lambda= e^{-\vec{H}} (d_q \f_{q_0})$ is not unit; the associated geodesic is not length-parametrized and has speed $\|\dot\gamma(t)\|^2 = 2H(\lambda) = \d(q_0,q)^2$. In this proof, with no risk of confusion, we use the symbol $\nabla h$ to denote the horizontal gradient $\grad(h)$ of $h \in C^\infty(M)$.

We drop $q_0$ from the notation of $\f_{q_0}$, since it is fixed. For any $p \in \Sigma_{q_0}$, the two curves
\begin{equation}
e^{\varepsilon \nabla \f}(p), \qquad \text{and} \qquad \pi \circ e^{\varepsilon \vec{H}}(d_p \f),
\end{equation}
define the same tangent vector at $p$. Hence we can exchange them at first order in $\varepsilon$. Let $d \f : \Sigma_{q_0} \to T^*M$ be the smooth map $p \mapsto d_p \f$. In particular, for any tensor $\eta$
\begin{equation}\label{eq:scambio}
\left.\frac{d}{d \varepsilon}\right|_{\varepsilon = 0} (e^{\varepsilon \nabla \f})^* \eta = \left.\frac{d}{d \varepsilon}\right|_{\varepsilon = 0} (\pi \circ e^{\vec{H}} \circ d\f)^* \eta.
\end{equation}
By definition of sub-Laplacian associated with a smooth volume $\omega$ we have
\begin{equation}
\Delta_{\omega}h =\frac{1}{\omega} \mathcal{L}_{\nabla h} \omega= \frac{1}{\omega} \left.\frac{d}{d\varepsilon}\right|_{\varepsilon=0} (e^{\varepsilon \nabla h})^* \omega, \qquad h \in C^\infty(M),
\end{equation}
where $e^{\tau X}$ denotes the flow of the vector field $X$. For $h = \f$, and using~\eqref{eq:scambio}, we obtain
\begin{equation}\label{eq:laplacianformula}
(\Delta_{\omega} \f)(q) =\frac{1}{\omega(W_1,\ldots,W_n)} \left.\frac{d}{d\varepsilon}\right|_{\varepsilon=0}\omega(\pi_* \circ e^{\varepsilon \vec{H}}_*\circ (d \f)_* (W_1,\ldots,W_n)),
\end{equation}
for any set of vectors $W_1,\ldots,W_n \in T_q M$. Consider a canonical frame $\{E_i(t),F_i(t)\}_{i=1}^n$ along the extremal $\lambda(t)$ as in Sec.~\ref{s:jac}, and the corresponding frame $f_i(t)= \pi_*F_i(t)$ along $\gamma(t)$. We will soon set $W_i = f_i(1)$ in~\eqref{eq:laplacianformula}. For any $q \in \Sigma_{q_0}$, we have $\pi \circ e^{-\vec{H}} (d_q \f) = q_0$, hence
\begin{equation}
\pi_* \circ e^{-\vec{H}}_* \circ (d \f)_*  = 0.
\end{equation}
In particular, since $\ker \pi_*|_{\lambda} = \spn\{E_1(0),\ldots,E_n(0)\}$, for all $i=1,\ldots,n$ we have
\begin{equation}
e^{-\vec{H}}_* \circ (d \f)_* f_i(1) =  \sum_{j=1}^n \Theta_{ji} E_j(0), \quad \Rightarrow \quad (d\f)_* f_i(1) = \sum_{j=1}^n \Theta_{ji} e^{\vec{H}}_* E_j(0),
\end{equation}
for some $n\times n$ matrix $\Theta$. The vector field $\J_j(t) = e^{t\vec{H}}_* E_j(0)$ is a Jacobi field along $\lambda(t)$ with initial condition $\J_j(0) = E_j(0)$. In particular, its components $\J_j(t) =\sum_{\ell=1}^n M_{\ell j}(t)E_\ell(t) + N_{\ell j}(t) F_\ell(t)$ solve~\eqref{eq:Jacobimatrix}. Moreover, since $\pi \circ d\f = \mathbb{I}$ on $\Sigma_{q_0}$, we have
\begin{align}
f_i(1) & = \pi_* \circ (d\f)_* f_i(1)  \\
& = \pi_* \sum_{\ell,j=1}^n \Theta_{j i}\left( M_{\ell j}(1)E_\ell (1) + N_{\ell j}(1) F_\ell(1)\right)  = \sum_{\ell=1}^n [N(1)\Theta]_{\ell i} f_\ell(1).
\end{align}
In particular $\Theta = N(1)^{-1}$. Hence
\begin{align}
\pi_* \circ e^{\varepsilon \vec{H}}_* \circ (d\f)_* f_i(1) =\sum_{\ell =1}^n [N(1+\varepsilon)\Theta]_{\ell i} f_\ell(1+ \varepsilon)  =\sum_{\ell=1}^n  [N(1+ \varepsilon) N(1)^{-1}]_{\ell i}  f_\ell(1+ \varepsilon).
\end{align}
Plugging this back into~\eqref{eq:laplacianformula}, we obtain
\begin{align}
(\Delta_{\omega} \f)(q) &= \frac{1}{\omega(f_1(1),\ldots,f_n(1))}\left.\frac{d}{d\varepsilon}\right|_{\varepsilon=0} \frac{\det(N(1+\varepsilon))}{\det N(1)} \omega(f_1(1+\varepsilon),\ldots,f_n(1+\varepsilon)) \\
& = \left.\frac{d}{d t}\right|_{t=1} \log (|\det N(t) \omega(f_1(t),\ldots,f_n(t))|) \\
& = \trace(\dot{N}(1) N(1)^{-1}) + \rho_\omega(d_q \f),
\end{align}
where we used the definition of canonical volume derivative, and Remark~\ref{rmk:traslation}. The matrix $N(t)$ solves~\eqref{eq:Jacobimatrix}, thus by the same splitting and notation of the previous proofs
\begin{align*}
(\Delta_\omega \f)(q) & = \trace(B V(1) + A) + \rho_\omega(d_q \f) & \text{by \eqref{eq:Jacobimatrix}}\\
& = \trace(B_{\mathrm{I}}V_{\mathrm{I}}(1)) + \trace(V_{\mathrm{II}}(1)) + \rho_\omega(d_q \f) & \text{by \eqref{eq:splitI-IIa} \eqref{eq:splitI-IIb} \eqref{eq:splitI-IIc}} \\
& = \trace(B_{\mathrm{I}}V_{\mathrm{I}}(1)) + \trace(V^\prime_{\mathrm{II}}(1)) + v^0_{\mathrm{II}}(1) + \rho_\omega(d_q \f) & \text{by  \eqref{eq:splitII-II'}}\\
& =(n-k) \trace(b_{\mathrm{I}} v_{\mathrm{I}}(1)) + (2k-n-1)v^\prime_{\mathrm{II}}(1) + v^0_{\mathrm{II}}(1)+ \rho_\omega(d_q \f), & \text{by \eqref{eq:smallvdef} \eqref{eq:smallabdef} \eqref{eq:smallv'def}}
\end{align*}
where $V(t)$ is the solution of~\eqref{eq:riccati} with curvature matrix associated with the extremal $\lambda(t)=e^{(t -1)\vec{H}}(d_q \f)$. We rescale $\lambda(t)$. Set $t_q := \d(q_0,q)$ and denote with $\bar{\lambda}(t) := e^{t \vec{H}}(\bar\lambda)$ the extremal with unit initial covector $\bar\lambda := \lambda / t_q$. By homogeneity of the Hamiltonian we have
\begin{equation}
\lambda(t) = e^{t \vec{H}}(\lambda) = e^{t \vec{H}}(t_q \bar{\lambda}) = t_q \bar\lambda(t_q t).
\end{equation}
By Remark~\ref{r:homgeneity-vol}, and the hypothesis on the canonical volume derivative, we have
\begin{equation}
\rho_\omega(d_q \f) = \rho_\omega(\lambda(1)) = \rho_\omega(t_q \bar\lambda(t_q)) = t_q \rho_\omega(\bar\lambda(t_q)) \leq t_q \kappa_\omega(\bar\lambda).
\end{equation}
By hypothesis $\Riccan^\alpha(\bar{\lambda}(t)) \geq \kappa_\alpha(\bar{\lambda})$ for all unit covectors $\bar{\lambda}$, and $\alpha=a,b,c$. Then by Remark~\ref{rmk:homogeneityproperties}
\begin{align}
\Riccan^a(\lambda(t)) & = t_q^{4} \Riccan^a(\bar\lambda(t_q t)) \geq t_q^4\kappa_a(\bar\lambda), \\
\Riccan^b(\lambda(t)) & = t_q^{2} \Riccan^b(\bar\lambda(t_q t)) \geq t_q^2\kappa_b(\bar\lambda),\\
\Riccan^c(\lambda(t)) & = t_q^{2} \Riccan^c(\bar\lambda(t_q t)) \geq t_q^2\kappa_c(\bar\lambda).
\end{align}
By Riccati comparison, as in the previous sections (and taking in account rescaling) we have 
\begin{equation}
v_{\mathrm{I}}(t)  \leq v_{t_q^4\kappa_a(\bar{\lambda}),t_q^2\kappa_b(\bar{\lambda})}(t), \qquad v^\prime_{\mathrm{II}}(t) \leq v_{t_q^2 \kappa_c(\bar{\lambda})}(t), \qquad v^0_{\mathrm{II}}(t)  \leq 1/t,
\end{equation}
for at least all $t \leq 1$. From the definition of the functions $s_{\kappa_a,\kappa_b}(t)$, $s_{\kappa_c}(t)$ and their homogeneity properties (see Remarks ~\ref{rmk:defskc} and~\ref{rmk:defskakb}) we obtain
\begin{align}
(\Delta_\omega \f)(q) & \leq (n-k)s_{t_q^4 \kappa_a(\bar\lambda),t_q^2 \kappa_b(\bar\lambda)}(1) + (2k-n-1)s_{t_q^2 \kappa_c(\bar\lambda)}(1) + 1 +t_q \kappa_\omega(\bar\lambda) \\
&  \leq  (n-k) t_q s_{\kappa_a(\bar\lambda),\kappa_b(\bar\lambda)}(t_q) + (2k-n-1)t_q  s_{\kappa_c(\bar\lambda)}(t_q) + 1 +t_q \kappa_\omega(\bar\lambda).
\end{align}
To recover an analogous result for $r = \d(q_0,\cdot)$ notice that $\nabla \f = r \nabla r$. Hence
\begin{equation}
\Delta_\omega \f = \dive_\omega(\nabla \f) = \dive_\omega(r \nabla r) = r \dive_\omega(\nabla r) + dr (\nabla r) = r \Delta_\omega r + 1.
\end{equation}
In particular, observing that $t_q = r(q)$, we have
\begin{equation}
(\Delta_\omega r)(q) =\frac{(\Delta_\omega \f)(q)-1}{r(q)} \leq (n-k) s_{\kappa_a(\bar{\lambda}),\kappa_b(\bar\lambda)}(r(q))+(2k-n-1) s_{\kappa_c(\bar\lambda)}(r(q)) + \kappa_\omega(\bar\lambda).
\end{equation}
To obtain the exact statement of Theorem~\ref{t:lapl-intro}, observe that the covector 
\begin{equation}
\bar{\lambda} = \lambda/t_q = e^{-\vec{H}}(d_q \f)/t_q = e^{-\vec{H}}(t_q d_q r)/t_q = e^{-t_q \vec{H}}(d_q r) = \lambda_{q_0}^q,
\end{equation}
is the initial covector of the unique length-parametrized geodesic joining $q_0$ with $q$. \hfill $\qed$

\subsection{Proof of Proposition \ref{p:3sasBM-2}}

We consider a sub-Riemannian length-parametrized geodesic $\gamma(t)$ and apply Theorem \ref{t:conjI-intro}. Then we study the maximum of $\bar{t}(\kappa_a,\kappa_b)$ over all geodesics. We use the expressions for the Ricci curvature of $3$-Sasakian manifold of Theorem~\ref{t:ricci-3-sas-intro}. 

In particular, under the assumption~\eqref{eq:boundrhoa}, $\varrho^a(v) \geq \sum_{\alpha} K \|Z_\alpha\|^2 = 2 K \|v\|^2$. Then set
\begin{align}
\kappa_b(v) & := 4+5\|v\|^2 = \tfrac{1}{3}\mathfrak{Ric}^b, \\
\kappa_a(v) & :=\|v\|^2\left(\tfrac{3}{2}K-\tfrac{7}{2}-\tfrac{15}{8}\|v\|^2\right) \le \tfrac{1}{3}\mathfrak{Ric}^a.  
\end{align}
Since $\kappa_b(v)>0$, conditions \eqref{eq:conditions} are equivalent to
\begin{equation}
	\frac{35}{2}\|v\|^4+(26+6K)\|v\|^2+16>0,
\end{equation}
which is positive for all $\|v\|\ge 0$ if $K\ge -1$. From Theorem \ref{t:conjI-intro} we get $t_*(\gamma)\le\bar{t}(v):=\bar{t}(\kappa_a(v),\kappa_b(v))$. Observe that to larger values of $K$ (and fixed $v$) correspond larger values of $\kappa_a(v)$, hence smaller blow-up times (by Riccati comparison, see Corollary \ref{c:riccaticomparison}). Thus, it is sufficient to prove the bound $\bar{t}(v) \leq \pi$ for fixed $K = -1$. From Proposition \ref{p:stima} we get 
\begin{equation}\label{eq:estime-t(v)}
\bar{t}(v)=\bar{t}(\kappa_a(v),\kappa_b(v))\le \frac{\pi}{\theta_-(v)}, \qquad \theta_\pm(v):=\frac{1}{2}(\sqrt{x+y}\pm\sqrt{x-y}),
\end{equation}
where $x=\frac{\kappa_b(v)}{2}$ and $y=\frac{\sqrt{\kappa_b^2(v)+4\kappa_a(v)}}{2}$. If $\|v\|=0$ then $\kappa_a(v) = 0$, $\kappa_b(v) =4$ and $\bar{t}(0,4) = \pi$. If $\|v\| >0$, one can check that we are in the case 1 of the proof of Proposition~\ref{p:stima}. In particular $\theta_+>\theta_->0$ are reals. Thus, as in~\eqref{eq:sinc-sinc}, $\bar{t}(v)$ is the first positive zero of
\begin{equation}
\chi_v(t):=\sinc(\theta_-(v)t)^2 - \sinc(\theta_+(v)t)^2.
\end{equation}
When $\theta_-(v)>1$, that is $\|v\|>\rho:=\sqrt{8/7}$, then $\bar{t}(v)<\pi$, by \eqref{eq:estime-t(v)}. 

On the other hand, one can check that if $0<\|v\|\le\rho$ then $\chi_v(\pi)\le 0$.  Since $\chi_v'(0)=0$ and $\chi_v''(0)=\tfrac{2}{3}(\theta_+^2-\theta_-^2)>0$, we conclude that also in this case $\bar{t}(v) \leq \pi$.
 \hfill $\qed$

\section{Sub-Riemannian geometry of 3-Sasakian manifolds} \label{s:3-Sas}

\subsection{Contact structures}
We collect here some results from the monograph \cite[Chapters 3,4,6,14]{blair} to which we refer for further details. Let $M$ be an odd-dimensional manifold, $\phi : \Gamma(TM) \to \Gamma(TM)$ be a $(1,1)$ tensor, $\xi \in \Gamma(TM)$ be a vector field and $\eta \in \Lambda^1 M$ be a one-form.	We say that $(\phi,\xi,\eta)$ is an \emph{almost contact structure} on $M$ if
	\begin{equation}\label{eq:acs}
		\phi^2 = -\mathbb{I} +\eta \otimes \xi, \qquad \eta(\xi) = 1.
	\end{equation}
This implies $\phi \xi = 0$ and $\eta \circ \phi = 0$.	We say that $g$ is a \emph{compatible metric} if
	\begin{equation}
		g(\phi X, \phi Y) = g(X,Y) - \eta(X)\eta(Y).
	\end{equation}
In this case, $(\phi,\xi,\eta,g)$ defines an \emph{almost contact metric structure} on $M$. Moreover, a compatible metric $g$ is an \emph{associated metric} if\footnote{The exterior differential is defined with the convention $d\eta(X,Y) = X(\eta(Y)) - Y(\eta(X)) - \eta([X,Y])$ for any one-form $\eta$. In \cite{blair}, the author uses a different convention, i.e. $2d\eta(X,Y) = X(\eta(Y)) - Y(\eta(X)) - \eta([X,Y])$, but there is no factor $2$ in~\eqref{eq:associated}. For this reason, our definitions agree with the ones of \cite{blair}.}  
\begin{equation}\label{eq:associated}
2g(X,\phi Y) = d\eta(X,Y).
\end{equation}	
In this case, $(\phi,\xi,\eta,g)$ is called a \emph{contact metric structure} on $M$.

\subsubsection{Sasakian structures}
Let $(\phi,\xi,\eta,g)$ be a (almost) contact metric structure on $M$, and consider the manifold $M\times\R$. We denote vector fields on $M\times \R$ by $(X,f\partial_t)$, where $X$ is tangent to $M$ and $t$ is the coordinate on $\R$. Define the $(1,1)$ tensor
\begin{equation}
\mathbf{J}(X,f\partial_t) = (\phi X-f\xi,\eta(X)\partial_t).
\end{equation}
Indeed $\mathbf{J}^2 = -\mathbb{I}$ and it defines an almost complex structure on $M\times \R$ (this was not possible on the odd-dimensional manifold $M$). We say that the (almost) contact metric structure $(\phi,\xi,\eta,g)$ is \emph{Sasakian} if the almost complex structure $\textbf{J}$ is a complex one. A celebrated theorem by Newlander and Nirenberg states that this is equivalent to the vanishing of the Nijenhuis tensor of $\mathbf{J}$. For a $(1,1)$ tensor $T$, the Nijenhuis $(2,1)$ tensor $[T,T]$ is
\begin{equation}
[T,T](X,Y) := T^2[X,Y] + [TX,TY] - T[TX,Y] - T[X,T Y].
\end{equation}
In terms of the original structure, the integrability condition $[\mathbf{J},\mathbf{J}] =0$ is equivalent to
\begin{equation}
[\phi,\phi](X,Y) + d\eta(X,Y) \xi = 0.
\end{equation}
Any Sasakian structure is \emph{$K$-type}, i.e. the \emph{Reeb vector field} $\xi$ is Killing: $\mathcal{L}_\xi g = 0$. The converse, however, is not true (except for $\dim M = 3$). Moreover, Sasakian structures are automatically contact metric structures, i.e. Sasakian implies~\eqref{eq:associated}. In particular the following is an equivalent characterization of Sasakian structures.
\begin{theorem}
An almost contact metric structure $(\phi,\xi,\eta,g)$ is Sasakian if and only if
\begin{equation}
(\nabla_X \phi) Y = g(X,Y)\xi - \eta(Y)X
\end{equation}
for all vector fields $X,Y \in \Gamma(TM)$. This directly implies
\begin{equation}
\nabla_Y \xi = - \phi Y.
\end{equation}
\end{theorem}

\subsection{Contact 3-structures}

Let $\dim M = 4d+3$. An \emph{almost contact $3$-structure} on $M$ is a collection of three distinct almost contact structures $(\phi_\alpha,\eta_\alpha,\xi_\alpha)$, where $\alpha = I,J,K$, that satisfy the following quaternionic-like compatibility relations 
	\begin{gather}
	\phi_K  = \phi_I\phi_J - \eta_J\otimes \xi_I = -\phi_J\phi_I + \eta_I\otimes \xi_J,  \label{eq:quat1}\\
	\xi_K  = \phi_I\xi_J = -\phi_J\xi_I, \qquad \eta_K = \eta_I\circ \phi_J = -\eta_J \circ \phi_I, \label{eq:quat2}
	\end{gather}
	for any even permutation of $I,J,K$. There always exists a metric $g$ on $M$ compatible with each structure. In this case $\{\phi_\alpha,\eta_\alpha,\xi_\alpha,g\}_{\alpha}$ is called an \emph{almost contact metric $3$-structure} on $M$. In particular $\xi_I,\xi_J,\xi_K$ are an orthonormal triple and
	\begin{equation}
	[\xi_I,\xi_J] = 2\xi_K,
	\end{equation}
and  analogously for cyclic permutations.
\begin{rmk}
Why $3$-structures? Given two almost contact structures satisfying (partial) qua\-ter\-nionic relations as~\eqref{eq:quat1}-\eqref{eq:quat2}, one can always define a third one to complete it to a almost contact $3$-structure. On the other hand an almost contact $3$-structure cannot be extended to include a fourth one (see \cite[Chapter 14]{blair}).
\end{rmk}

\subsection{3-Sasakian manifolds}
If each almost contact metric structure $(\phi_\alpha,\eta_\alpha,\xi_\alpha,g)$ is actually a contact metric structure (i.e.~\eqref{eq:associated} holds), we say that $\{\phi_\alpha,\eta_\alpha,\xi_\alpha,g\}_{\alpha}$ is a \emph{contact metric $3$-structure}. By a result of Kashiwada~\cite{Kashiwada-contact3}, each $(\phi_\alpha,\eta_\alpha,\xi_\alpha,g)$ is actually Sasakian. In this case, we say that $M$ with the structure $\{\phi_\alpha,\eta_\alpha,\xi_\alpha,g\}$ is a \emph{$3$-Sasakian manifold}.

\subsubsection{Quaternionic indices notation}\label{s:quatindices}
We can collect all the relations on a $3$-Sasakian structure with the following notation. If $\alpha,\beta  =I,J,K$
\begin{gather}
\phi_{\alpha\beta} = \phi_\alpha\phi_\beta - \eta_\beta \otimes \xi_\alpha, \\
\xi_{\alpha\beta}  = \phi_\alpha \xi_\beta , \qquad \eta_{\alpha\beta} = \eta_\alpha \circ \phi_{\beta},
\end{gather}
where the product $\alpha\beta$ denotes the quaternionic product and we use the conventions $\phi_{\pm 1} = \pm \mathbb{I}$, $\eta_1 =0$, $\xi_1 = 0$ and $\phi_{-\alpha} = - \phi_{\alpha}$. 
Moreover, we recall the Sasakian properties
\begin{align}
(\nabla_Y \phi_\alpha) Z  = g (Y,Z) \xi_\alpha - \eta_\alpha(Z) Y, \qquad \text{and} \qquad \nabla_Y \xi_\alpha  = - \phi_\alpha Y.
\end{align}
for all $X,Y,Z \in \Gamma(TM)$ and $\alpha = I,J,K$.

The following result is proved in \cite[Thm. A]{BGM-3-Sasakian}, to which we refer for details.
\begin{theorem}\label{t:BGM}
Let $\{\phi_\alpha,\eta_\alpha,\xi_\alpha,g\}_\alpha$ be a $3$-Sasakian structure on a smooth manifold $M$ of dimension $4d+3$. Assume that the Killing vector fields $\xi_\alpha$ are complete for $\alpha=1,2,3$. Then
\begin{itemize}
\item[(i)] $(M,g)$ is an Einstein manifold of positive scalar curvature equal to $(4d+2)(4d+3)$;
\item[(ii)] The metric $g$ is bundle-like with respect to the foliation $\mathcal{F}$ defined by $\{\xi_I,\xi_J,\xi_K\}$;
\item[(iii)] Each leaf of the foliation $\mathcal{F}$ is a $3$-dimensional homogeneous spherical space form;
\item[(iv)] The space of leaves $M/\mathcal{F}$ is a quaternionic K\"ahler orbifold of dimension $4d$ with positive scalar curvature equal to $16d(d+2)$.
\end{itemize}
Hence, every complete $3$-Sasakian manifold is compact with finite fundamental group and Riemannian diameter less than or equal to $\pi$.
\end{theorem}
We stress that, even if the \emph{Riemannian} diameter of a $3$-Sasakian manifold is bounded by the classical Bonnet-Myers theorem, nothing is known about the \emph{sub-Riemannian} one. In fact, a priori, sub-Riemannian distances are larger then Riemannian ones. 

\subsubsection{Some curvature properties of 3-Sasakian manifolds}
We will need the following results about the Riemannian curvature of $3$-Sasakian structures, proved in \cite[Prop. 3.2]{Tanno} and \cite[Prop. 2.17]{BGM-3-Sasakian}, respectively. Here $\distr$ is the orthogonal complement to $\spn\{\xi_I,\xi_J,\xi_K\}$ w.r.t. the Riemannian metric $g$ and $\Sec$ is the sectional curvature of the Riemannian structure.
\begin{proposition}\label{p:sumholo}
For any $X \in \distr_q$, the sum of the $\phi_\alpha$-sectional curvatures is constant:
\begin{equation}
\Sec(X,\phi_I X)+\Sec(X,\phi_J X)+\Sec(X,\phi_K X)= 3.
\end{equation}
\end{proposition}
\begin{proposition}\label{p:transversecurv}
For any $X \in \distr_q$ we have $\Sec(X,\xi_\alpha)=1$ for all $\alpha=I,J,K$.
\end{proposition}

\subsection{Sub-Riemannian geometry of 3-Sasakian manifolds}\label{s:srgeo3sas}
Any $3$-Sasakian structure $\{\phi_\alpha,\eta_\alpha,\xi_\alpha,g\}_\alpha$ carries a natural sub-Riemannian structure. The distribution $\distr \subset TM$ is
\begin{equation}
\distr := \bigcap_{\alpha =I,J,K} \ker \eta_{\alpha}.
\end{equation}
Indeed $\distr$ is a corank $3$ sub-bundle, orthogonal to $\xi_I,\xi_J,\xi_K$. One can check that $\distr$ is a fat distribution, thus the restriction of $g$ to $\distr$ is a fat sub-Riemannian structure on $M$.

\begin{lemma}\label{l:costanza}
Let $\lambda \in T^*M$ be the initial covector of the extremal $\lambda(t) = e^{t\vec{H}}(\lambda)$. Let $v_\alpha(\lambda):=\langle \lambda, \xi_\alpha\rangle$ smooth functions on $T^*M$ for $\alpha = I,J,K$. Then $v_\alpha$ is constant along $\lambda(t)$.
\end{lemma}
\begin{proof}
Let $X_1,\ldots,X_{4d}$ be a local orthonormal frame for $\distr$ around $\gamma(t)$. The Hamiltonian is $H = \tfrac{1}{2}\sum_{i=1}^{d} u_i^2$, where $u_i(\lambda):= \langle \lambda, X_i\rangle$, for $i=1,\ldots,4d$. Using Hamilton equations
\begin{equation}\label{eq:costanza}
\dot{v}_\alpha = \{H,v_\alpha \} = \sum_{i=1}^{d} u_i \{u_i, v_\alpha \} = \sum_{i,j=1}^{d} u_i u_j g([X_i,\xi_\alpha],X_j) + \sum_{i=1}^{d} \sum_{\beta} v_\beta u_i g([X_i,\xi_\alpha],\xi_\beta).
\end{equation}
Observe that
\begin{equation}
\eta_{\beta}([X_i,\xi_\alpha]) = -d\eta_{\beta}(X_i,\xi_\alpha) = -2g(X_i,\phi_\beta \xi_\alpha) = 0.
\end{equation}
Hence $[X_i,\xi_\alpha] \in \distr$ and the second term in~\eqref{eq:costanza}  vanishes. Moreover each contact structure $(\eta_\alpha,\phi_\alpha,\xi_\alpha)$ is K-type, that is $\mathcal{L}_{\xi_\alpha} g = 0$. This implies that the matrix $g([\xi_\alpha,X_i],X_j)$ is skew-symmetric (for any fixed $\alpha$). Then also first term of~\eqref{eq:costanza} vanishes.
\end{proof}

The next proposition can serve, alternatively, as the definition of Popp volume on $3$-Sasakian structures. We refer the reader interested in the general definition to~\cite{BR-Popp}.
\begin{proposition}\label{p:Popp3Sas}
Up to a constant factor, the Popp volume of the sub-Riemannian structure of a $3$-Sasakian manifold is proportional to the Riemannian one.
\end{proposition}
\begin{proof}
Let $\omega \in \Lambda^{n}(M)$ be the Popp volume. The explicit formula in~\cite{BR-Popp} gives
\begin{equation}
\omega(X_1,\ldots,X_{4d},\xi_I,\xi_J,\xi_K) = \frac{1}{\sqrt{\det (B)}},
\end{equation}
for any local orthonormal frame $X_1,\ldots,X_{4d}$ of $\distr$, where $B$ is the matrix with components
\begin{equation}
B_{\alpha\beta} := \sum_{i,j=1}^{4d}\eta_\alpha([X_i,X_j]) \eta_\beta([X_i,X_j]) = 4\trace(\phi_\alpha \phi_\beta^*), \qquad \alpha,\beta = I,J,K,
\end{equation}
where we used the properties of $3$-Sasakian structures. In particular $\det (B) = 12^3$.
\end{proof}
\begin{rmk}
Scaling a volume by a constant factor does not change the associated divergence operator. Hence the sub-Laplacian associated with Popp volume coincides, up to a sign, with the sub-Laplacian used in quaternionic contact geometry (see, for example, \cite{IMV-QCYamabe,iv15}). 
\end{rmk}

\begin{example}[The quaternionic Hopf fibration]
The field of quaternions is
\begin{equation}
\H = \{q = x+ Iy + Jz + K w \mid (x,y,z,w) \in \R^4 \},
\end{equation}
with norm $\|q\|^2 = x^2 + y^2 + z^2 + w^2$. The tangent spaces $T_q \H \simeq \H$ have a natural structure of $\H$-module. With this identification, the multiplication by $I,J,K$ induces the complex structures $\Phi_I,\Phi_J,\Phi_K : T \H \to T \H$. In real coordinates
\begin{align}
\Phi_I \partial_x & = + \partial_y, & \Phi_J \partial_x & = +\partial_z, &  \Phi_K \partial_x & = +\partial_w, \\
\Phi_I \partial_y & = -\partial_x, & \Phi_J \partial_y & = -\partial_w, &  \Phi_K \partial_y & = + \partial_z, \\
\Phi_I \partial_z & = +\partial_w, & \Phi_J \partial_z & = -\partial_x, &  \Phi_K \partial_z & = -\partial_y, \\
\Phi_I \partial_w & = -\partial_z, & \Phi_J \partial_w & = +\partial_y, &  \Phi_K \partial_w & = -\partial_x. 
\end{align}
The \emph{quaternionic unit sphere} is the real manifold of dimension $4d+3$
\begin{equation}
\mathbb{S}^{4d+3}=\left\lbrace q = (q_1,\ldots,q_{d+1}) \in \H^{d+1} \mid \|q\|=1 \right\rbrace,
\end{equation}
equipped with the standard round metric $g$. The inward unit normal vector is
\begin{equation}
\mathbf{n} = -\sum_{i=1}^{d+1} x_i\partial_{x_i} +y_i\partial_{y_i} + z_i\partial_{z_i} + w_i\partial_{w_i}.
\end{equation}
The vectors $\xi_{\alpha}:= \Phi_\alpha \mathbf{n}$ are tangent to $\mathbb{S}^{4d+3}$ and are given by
\begin{align}
\xi_I  & = \sum_{i=1}^{d+1}  y_i\partial_{x_i} - x_i\partial_{y_i}  +w_i\partial_{z_i} -z_i\partial_{w_i},\\
\xi_J & = \sum_{i=1}^{d+1}    z_i\partial_{x_i} -w_i\partial_{y_i} -x_i\partial_{z_i} +y_i\partial_{w_i},\\
\xi_K & = \sum_{i=1}^{d+1} w_i\partial_{x_i} + z_i\partial_{y_i}  -y_i\partial_{z_i} - x_i\partial_{w_i}.   
\end{align}
Consider the three one-forms
\begin{align}
\eta_I  & = \sum_{i=1}^{d+1}  y_id x_i  -x_i d y_i  + w_id z_i -z_id w_i,\\
\eta_J & = \sum_{i=1}^{d+1}    z_id x_i - w_id y_i-x_id z_i +y_id w_i,\\
\eta_K & = \sum_{i=1}^{d+1}  w_id x_i + z_id y_i  -y_id z_i - x_id w_i.   
\end{align}
The three almost complex structures on $\mathbb{S}^{4d+3}$ are defined as $\phi_\alpha:= \mathrm{pr} \circ\Phi_\alpha$, for $\alpha = I,J,K$, where $\mathrm{pr}$ is the orthogonal projection on the sphere. One can check that the restrictions of $(\phi_\alpha,\eta_\alpha,\xi_\alpha,g)$ to $\mathbb{S}^{4d+3}$ define a $3$-Sasakian structure on it.

The natural action of the unit quaternions $\{p \in \H \mid \|p \|=1\}=\mathbb{S}^3 \simeq \mathrm{SU}(2)$ on $\mathbb{S}^{4d+3}$ is
\begin{equation}
p \cdot (q_1,\ldots,q_{d+1}) = (p q_1,\ldots, p q_{d+1}).
\end{equation}
The projection $\pi$ on the quotient $\mathbb{HP}^d$ is the so-called \emph{quaternionic Hopf fibration}:
\begin{equation}
\mathbb{S}^3 \hookrightarrow \mathbb{S}^{4d+3} \xrightarrow{\pi} \mathbb{HP}^{d}.
\end{equation}
The vector fields $\xi_\alpha$ generated by the action of $e^{\varepsilon I},e^{\varepsilon J},e^{\varepsilon K}$ on $\mathbb{S}^{4d+3}$ are tangent to the fibers.
\end{example}

\section{Computation of curvature and canonical frame for 3-Sasakian manifolds} \label{s:CanFrame}

Fix a $3$-Sasakian manifold $M$ of dimension $n=4d+3$, and consider its sub-Riemannian structure as in Section~\ref{s:srgeo3sas}, with $k = \rank\distr = 3$. We compute the canonical frame along an extremal $\lambda(t)$ (for small $t$) with initial covector $\lambda \in U^*M$, and the Ricci curvatures $\Riccan^\mu(\lambda(t))$. To do this, we exploit the auxiliary Riemannian structure $g$ of the $3$-Sasakian manifold. Hence $\nabla$ denotes the covariant derivative and $R^\nabla$ the Riemann curvature tensor w.r.t. the Levi-Civita connection. The formulas for the sub-Riemannian curvature will only depend on the sub-Riemannian structure $(M,\distr,g|_\distr)$. In the notation of Section~\ref{s:jac}, we split
\begin{equation}
E(t) = (E_a(t),E_b(t),E_c(t))^*,\qquad F(t) = (F_a(t),F_b(t),F_c(t))^*.
\end{equation}
where $E_\mu(t)$ is a $|\mu|$-tuple, with $\mu=a,b,c$, with $|a|=|b|=3$ and $|c| = 4d-3$. Moreover, we express the structural equations (Proposition \ref{p:can}) in the following explicit form:
\begin{align}
\dot{E}_a & = E_b,  	& \dot{E}_b & = -F_b,  	& \dot{E}_c & = -F_c,  \\
\dot{F}_a & =  \sum_{\mu=a,b,c}R_{a \mu}(t)E_{\mu}, &  \dot{F}_b & =  \sum_{\mu=a,b,c}R_{b\mu}(t)E_{\mu} - F_a & \dot{F}_c, & =  \sum_{\mu=a,b,c}R_{c\mu}(t)E_{\mu}.
\end{align}
where the curvature matrix $R(t) = R(t)^*$ is
\begin{equation}
R(t) = \begin{pmatrix}
R_{aa}(t) & R_{ab}(t) & R_{ac}(t) \\
R_{ba}(t) & R_{bb}(t) & R_{bc}(t) \\
R_{ca}(t) & R_{cb}(t) & R_{cc}(t)
\end{pmatrix},
\end{equation}
and satisfies the additional condition $R_{ab}(t) = - R_{ab}(t)^*$. We stress that $R(t)$ is a matrix representation of the curvature operator in the basis given by the projections $f_\mu(t)=\pi_* F_\mu(t)$ for $\mu=a,b,c$, but the Ricci curvatures do not depend on such a representation.

\subsection{Auxiliary frame}
We build a convenient local frame on $M$, associated with a given trajectory (the geodesic $\gamma(t) = \pi(\lambda(t))$, in our case).

\begin{lemma}\label{fermi_frame}
There exists a horizontal frame $X_i$, $i \in \{1,\dots,4d\}$, in a neighborhood of $\gamma(0)$, such that for all $\alpha\in\{I,J,K\}$ and $i,j\in\{1,\dots,4d\}$,
	\begin{itemize}
		\item the frame is orthonormal,
		\item $\nabla_{X_i} X_j|_{\gamma(t)} = \frac{1}{2}[X_i,X_j]^v_{\gamma(t)}$, where $^v$ denotes the orthogonal projection on $\distr^\perp$,
		\item $[\xi_\alpha, X_i]=0$.
	\end{itemize}
\end{lemma}

\begin{proof}
The transverse distribution generated by $\xi_I,\xi_J,\xi_K$ is involutive. Hence by Frobenius theorem there exists a neighborhood $\mathcal{O}$ of $\gamma(0)$ and a smooth submersion $\pi:\mathcal{O}\subset M \to \mathbb{R}^{4d}$ such that the fibers are the integral manifolds of the transverse distribution. We give $\bar{\mathcal{O}} = \pi(\mathcal{O})$ the Riemannian metric such that $\pi:\mathcal{O}\to \bar{\mathcal{O}}$ is a Riemannian submersion (w.r.t. the Riemannian structure of the $3$-Sasakian manifold). Let $\bar{\nabla}$ be the covariant derivative on $\bar{\mathcal{O}}$.

We consider on $\bar{\mathcal{O}}$, an orthonormal frame $\{\bar{X}_1,\dots,\bar{X}_{4d}\}$, such that $\bar\nabla_{\bar{X}_i}\bar X_j|_{\bar\gamma(t)} = 0$ for $t$ small enough. The existence of this frame is proved in \cite[Thm. 3.1]{HandbookNORMALFRAMES}, with a construction inspired by Fermi normal coordinates \cite{Mis-FermiNormal}. Since $\pi : \mathcal{O} \to \bar{\mathcal{O}}$ is a Riemannian submersion, we can lift the frame $\bar{X}_i$ to a horizontal orthonormal frame $X_i \in \Gamma(\distr)$ on $\mathcal{O}$. Then by standard formulas \cite[Chap. 3.D]{GHL-RiemannianGeom} relating the covariant derivatives of a submersion, we obtain
	\begin{equation}
	\nabla_{X_i} X_j = \widetilde{\bar{\nabla}_{\bar X_i} \bar X_j} + \frac{1}{2}[X_i,X_j]^v,
	\end{equation}
	where the tilde denotes the horizontal lift. Finally, notice that $[\xi_\alpha, X_i] \in \Gamma(\distr)$ and also $\pi_*[\xi_\alpha, X_i]=[\pi_*\xi,\bar{X}_i]=0$, so $[\xi_\alpha,X_i]=0$.
\end{proof}
\begin{rmk}
The frame of Lemma~\ref{fermi_frame} is closely related with qc-normal frames. Qc-normal frames are defined for the general class of quaternionic contact (qc) manifold, and satisfy -- at a single point $q_0$ -- a series of conditions formulated in terms of of Biquard connection. Their existence is proved in \cite[Lemma 4.5]{IMV-QCYamabe}. Using the relation between Biquard and Levi-Civita connection \cite[Eq. 6.3]{IMV-QCEinsten}, one can show that, in the case of a 3-Sasakian manifold, the conditions satisfied by the frame of Lemma \ref{fermi_frame} are equivalent to the conditions defining a qc-frame at each point along the curve $\gamma(t)$ in a neighborhood of $\gamma(0)$. For this reason, one might call the frame of Lemma~\ref{fermi_frame} a \emph{qc-Fermi frame}. We also mention that the existence of a qc-Fermi frame (i.e. a qc-frame along a curve) has been proved for a general qc manifold in the recent paper \cite[Lemma 24]{BI-QC}.
\end{rmk}
\textbf{Notation and conventions:} 
\begin{itemize}
\item Latin indices $i,j,k,\dots$ belong to $\{1,\dots,4d\}$ and Greek ones $\alpha,\beta,\tau\dots$ are quaternions $\{I,J,K\}$. Repeated indices are summed over their maximal range;
\item We use the same quaternionic indices notation of Section~\ref{s:quatindices};
\item The dot denotes the Lie derivative in the direction of $\vec{H}$;
\item For $n$-tuples $v,w$ of vector fields along $\lambda(t)$, the symbol $\sigma(v,w)$ denotes the matrix $\sigma(v_i,w_j)$. Notice that $\sigma(v,w)^* = -\sigma(w,v)$ and that $\frac{d}{dt} \sigma(v,w) = \sigma(\dot{v},w) + \sigma(v,\dot{w})$;
\item For $n$-tuples of vectors $v$, and matrices $L$, the juxtaposition $Lv$ denotes the $n$-tuple of vectors obtained by matrix multiplication;
\item For functions $f,g \in C^\infty(T^*M)$, the symbol $\{f,g\}$ denotes the Poisson bracket. We make systematic use of symplectic calculus (see \cite{Agrachevbook} for reference).
\end{itemize}

\subsection{Hamiltonian frame}
Let us consider the \textit{momentum functions} $u_i,v_\alpha: T^*M \to \R$
\begin{align}
u_i & =  \langle \lambda, X_i \rangle , \qquad i = 1,\dots,4d, \\
v_\alpha & = \langle \lambda, \xi_\alpha\rangle , \qquad \alpha = I,J,K.
\end{align}
The momentum functions define coordinates $(u,v)$ on each fiber of $T^*M$. In turn, they define local vector fields $\partial_{v_\alpha}$ and $\partial_{u_i}$ on $T^*M$ (with the property that $\pi_* \partial_{v_\alpha} = \pi_*\partial_{u_i} = 0$). Moreover, they define also the Hamiltonian vector fields $\vec{u}_i$ and $\vec{v}_\alpha$. The \emph{hamiltonian frame} associated with $\{\xi_\alpha,X_i\}$ is the local frame on $T^*M$ around $\lambda(0)$ given by $\{\partial_{u_i},\partial_{v_\alpha},\vec{u}_i,\vec{v}_\alpha\}$. 

The following $3\times 3$ skew-symmetric matrix contains the ``vertical'' part of the covector:
\begin{equation}\label{eq:Vdef}
V_{\alpha\beta}:= v_{\alpha\beta}, \qquad \alpha,\beta = I,J,K.
\end{equation}
In the r.h.s. of~\eqref{eq:Vdef}, the notation $\alpha\beta$ denotes the product of quaternions with the convention $v_{\alpha^2}=-v_1=0$. Thus,~\eqref{eq:Vdef} is the standard identification $\mathbb{R}^3 \simeq \mathfrak{so}(3)$. The sub-Riemannian Hamiltonian and the corresponding Hamiltonian vector field are
\begin{equation}
H = \frac{1}{2} u_i u_i, \qquad \vec{H} = u_i \vec{u}_i.
\end{equation}

\begin{lemma} The momentum functions $u_i,v_\alpha$ have the following properties:
	\begin{itemize}
		\item[(1)] $\{u_i,v_\alpha\}=0$,
		\item[(2)] $\{v_\alpha,v_\beta\} = 2 v_{\alpha\beta}$,
		\item[(3)] $\{u_i,u_j\} = 2v_\alpha g(\phi_\alpha X_i,X_j) +u_k g( X_k,[X_i,X_j]) $.
	\end{itemize}
	Moreover, along the extremal $\lambda(t)$, we have
	\begin{itemize}
		\item[(4)] $\{u_i,u_j\} =  2v_\alpha g(\phi_\alpha X_i,X_j)$,
		\item[(5)] $\partial_{u_k} \{u_i,u_j\}=0$,
		\item[(6)] $\partial_{v_\alpha} \{u_i,u_j\}= 2g(\phi_\alpha X_i,X_j)$,
		\item[(7)] $\overrightarrow{\{u_i,u_j\}} = 2g(\phi_\alpha X_i,X_j)\vec{v}_\alpha +2v_\alpha g((\phi_{\alpha\tau}-\phi_{\tau\alpha}) X_i, X_j)\partial_{v_\tau} -u_k X_\ell g( X_k,[X_i,X_j]) \partial_{u_\ell}$.
	\end{itemize}
\end{lemma}
\begin{proof}
	Properties (1) and (2) follow from the definition of Poisson bracket and the fact that $[\xi_\alpha,X_i]=0$ and $[\xi_\alpha,\xi_\beta] = 2\xi_{\alpha\beta}$. For (3) we compute
\begin{align*}
	\{u_i,u_j\} & = v_\alpha g( \xi_\alpha,[X_i,X_j])  + u_k  g( X_k,[X_i,X_j])   \\
		& = -v_\alpha d\eta_\alpha(X_i,X_j) + u_k  g( X_k,[X_i,X_j])  = 2v_\alpha g(\phi_\alpha X_i, X_j) + u_k  g( X_k,[X_i,X_j]) .
\end{align*}
Point (4) follows from (3) and Lemma~\ref{fermi_frame}. Points (4)-(5) follow from (3). For (7)
\begin{align*}
\overrightarrow{\{u_i,u_j\}} & = 2\vec{v_\alpha} g(\phi_\alpha X_i,X_j) +\vec{u_k}\cancel{ g( X_k,[X_i,X_j]) } + 2v_\alpha \overrightarrow{g(\phi_\alpha X_i,X_j)} +u_k\overrightarrow{  g( X_k,[X_i,X_j]) }\\
& = 2g(\phi_\alpha X_i,X_j)\vec{v_\alpha} -2v_\alpha \cancel{X_\ell g(\phi_\alpha X_i,X_j)}\partial_{u_\ell} -2v_\alpha \xi_\tau g(\phi_\alpha X_i,X_j)\partial_{v_\tau} \\
& \quad -u_k X_\ell g( X_k,[X_i,X_j]) \partial_{u_\ell} -u_k \bcancel{\xi_\tau g( X_k,[X_i,X_j])} \partial_{v_\tau},
\end{align*}
where the first barred term vanishes by Lemma~\ref{fermi_frame}, the second one by direct computation and the last one by Jacobi identity and Lemma~\ref{fermi_frame}. To conclude, we observe that
\begin{align}
\xi_\tau g(\phi_\alpha X_i,  X_j) & =  \cancel{g((\nabla_{\xi_\tau} \phi_\alpha) X_i,  X_j)} +g( \phi_\alpha \nabla_{\xi_\tau} X_i,  X_j) + g( \phi_\alpha X_i,   \nabla_{\xi_\tau} X_j) \\
& = -g(\phi_\alpha  X_j,\nabla_{X_i} \xi_\tau) + g( \phi_\alpha X_i,   \nabla_{X_j} \xi_\tau) \\
& =  g(\phi_\alpha  X_j,\phi_\tau X_i) - g( \phi_\alpha X_i,   \phi_\tau X_j)  = g((\phi_{\alpha\tau}-\phi_{\tau\alpha}) X_i, X_j),
\end{align}
where the first barred term vanishes by Lemma~\ref{fermi_frame}.
\end{proof}

\begin{lemma}\label{l:secder} Let  $v_\alpha(t) =  \langle \lambda(t),\xi_\alpha|_{\gamma(t)}\rangle $, for $\alpha =I,J,K$. Then, along the geodesic, we have
	\begin{equation}
	\nabla_{\dot\gamma}\dot\gamma = 2 v_\alpha \phi_\alpha \dot\gamma,
	\end{equation}
\end{lemma}
\begin{proof}
	Indeed $\gamma(t) = u_i(t) X_i|_{\gamma(t)}$, with $u_i(t) = \langle \lambda(t),X_i|_{\gamma(t)}\rangle $. Then, suppressing $t$
	\begin{align}
	\nabla_{\dot\gamma}\dot\gamma & = \dot{u}_i X_i + u_i u_k \nabla_{X_k} X_i  = \{H,u_i\} X_i + \cancel{u_i u_k \tfrac{1}{2}[X_i,X_k]^v} \\
	& = u_k \{u_k ,u_i\} X_i  = 2u_k v_\alpha  g( \phi_\alpha X_k,X_i)  X_i = 2v_\alpha \phi_\alpha \dot\gamma,
	\end{align}
where the barred term vanishes by skew-symmetry.
\end{proof}

\begin{lemma}[Fundamental computations]\label{l:fund} Along the extremal, we have
	\begin{align}
	\dot{\partial}_v & = 2A\partial_u, & 
	\dot{\vec{u}} & = 2C\vec{u}-2A^*\vec{v}+B\partial_u+2D\partial_v, \\
	\dot{\partial}_u & = - \vec{u}, &
	\dot{\vec{v}} & = 0,
	\end{align}
	where we defined the following matrices, computed along the extremal:
	\begin{align}
	A_{\beta i} & := - g( \phi_\beta\dot{\gamma},X_i) ,   & \text{$3 \times 4d$ matrix}, \\
	B_{ij} & := R^{\nabla}(\dot{\gamma},X_i,X_j,\dot{\gamma}) + 3  g( X_i,\Pi_\phi X_j) , & \text{$4d\times 4d$ symmetric matrix},\\
	C_{ij} & := -v_\alpha g( \phi_\alpha X_i,X_j) , & \text{ $4d\times 4d$ skew-symmetric matrix}, \\
	D_{i\tau} &  :=  v_\alpha g( X_i, (\phi_{\alpha\tau}-\phi_{\tau\alpha})\dot{\gamma}) , & \text{$4d \times 3$ matrix},
	\end{align}
and $\Pi_\phi: \Gamma(\distr) \to \Gamma(\distr)$ is the orthogonal projection on $\spn\{\phi_I\dot\gamma,\phi_J\dot\gamma,\phi_K\dot\gamma\}.$
\end{lemma}	

\begin{proof} 
	By direct computations (along the extremal) we get
\begin{align*}
	\dot{\partial}_{v_\beta} &= 
	[u_j\vec{u}_j,\partial_{v_\beta}] =  -\partial_{v_\beta}(u_j)\vec{u}_j + u_j[\vec{u}_j,\partial_{v_\beta}] = u_j[\vec{u}_j,\partial_{v_\beta}](u_i)\partial_{u_i} + u_j[\vec{u}_j,\partial_{v_\beta}](v_\alpha)\partial_{v_\alpha}\\
	&= -u_j\partial_{v_\beta}\{u_j,u_i\}\partial_{u_i} -u_j \partial_{v_\beta}\cancel{\{u_j,v_\alpha\}} \partial_{v_\alpha}= -u_j g( 2\phi_\beta X_j,X_i) \partial_{u_i} = -2 g( \phi_\beta\dot\gamma,X_i) \partial_{u_i}.\\
	\dot{\partial}_{u_i} &= 
	[u_j\vec{u}_j,\partial_{u_i}] = -\partial_{u_i}(u_j)\vec{u}_j + u_j[\vec{u}_j,\partial_{u_i}] = -\vec{u}_i -u_j\cancel{\partial_{u_i}\{u_j,u_\ell\}}\partial_{u_\ell} = -\vec{u}_i.\\
	\dot{\vec{v}}_\beta &=	
	[u_j\vec{u}_j,\vec{v}_\beta] = -\vec{v}_\beta(u_j)\vec{u}_j + u_j[\vec{u}_j,\vec{v}_\beta] = -\{v_\beta,u_j\}\vec{u}_j + u_j\overrightarrow{\{u_j,v_\beta\}} = 0.\\
	\dot{\vec{u}}_i &=
	[u_j\vec{u}_j,\vec{u}_i] = -\vec{u}_i(u_j)\vec{u}_j + u_j[\vec{u}_j,\vec{u}_i] = -\{u_i,u_j\} \vec{u}_j + u_j\overrightarrow{\{u_j,u_i\}}\\
 &= -2v_\alpha g(\phi_\alpha X_i,X_j )  \vec{u}_j + 2 g(\phi_\alpha \dot\gamma,X_i)\vec{v}_\alpha +2  v_\alpha g((\phi_{\alpha\tau}-\phi_{\tau\alpha}) \dot\gamma, X_i)\partial_{v_\tau} \\
 & \quad - u_j u_k X_\ell g( X_k,[X_j,X_i]) \partial_{u_\ell}.
\end{align*}
To complete the proof, we show that $u_j u_k X_\ell g( X_k,[X_i,X_j])  = R(\dot\gamma,X_i,X_\ell,\dot\gamma) + 3 g(X_i, \Pi_\phi X_\ell)$.	From the definition of the Riemann curvature tensor, and Lemma~\ref{fermi_frame}, we have,
	\begin{align*}
	R^{\nabla}(\dot\gamma,X_i,X_\ell,\dot\gamma) & = u_k u_j g(\nabla_{X_k}\nabla_{X_i}X_\ell - \nabla_{X_i}\nabla_{X_k}X_\ell - \nabla_{[X_k,X_i]}X_\ell , X_j) \\
	& =  u_k u_j X_k g(\nabla_{X_i}X_\ell,X_j)  - u_k u_j g(\nabla_{X_i}X_\ell,\nabla_{X_k}X_j)  - u_k u_j X_i g(\nabla_{X_k}X_\ell,X_j) \\
	&\quad  + u_k u_j g(\nabla_{X_k}X_\ell,\nabla_{X_i}X_j)  - u_k u_j g([X_k,X_i],\xi_\tau)   g(\nabla_{\xi_\tau}X_\ell , X_j) .
	\end{align*}
	Notice that $u_k X_k g(\nabla_{X_i}X_\ell,X_i) =0$ since it is the derivative in the direction of $\dot\gamma(t)$ of $ g(\nabla_{X_i}X_\ell,X_j) |_{\gamma(t)} = 0$. On the other hand,	$ g(\nabla_{X_i}X_\ell,\nabla_{X_k}X_j) |_{\gamma(t)}$ is skew-symmetric w.r.t $k$ and $j$. Hence $u_k u_j  g(\nabla_{X_i}X_\ell,\nabla_{X_k}X_j) =0$. Thus
	\begin{align*}
	R^{\nabla}(\dot\gamma,X_i,X_\ell,\dot\gamma) & =
	- u_k u_j X_i  g(\nabla_{X_k}X_\ell,X_j)  +  u_k u_j  g(\nabla_{X_k}X_\ell,\nabla_{X_i}X_j)  \\
	& \quad - u_k u_j g([X_k,X_i],\xi_\tau)   g(\nabla_{\xi_\tau}X_\ell , X_j)  \\
	& = -\tfrac{1}{2}u_k u_j X_i \left(g([X_k,X_\ell],X_j)  +  g([X_j,X_k],X_\ell)  +  g([X_j,X_\ell],X_k) \right)\\
	&\quad + \tfrac{1}{4}u_k u_j g([X_k,X_\ell],\xi_\tau)   g(\xi_\tau,[X_i,X_j])  -2u_k u_j g( X_k,\phi_\tau X_i)   g( \phi_\tau X_\ell,X_j) \\
	& =-u_k u_j X_i g([X_k,X_\ell],X_j)  + u_k u_j g( X_k,\phi_\tau X_\ell)   g( X_i,\phi_\tau X_j)  \\
	& \quad - 2 g( \dot{\gamma},\phi_\tau X_i)   g(\phi_\tau X_\ell,\dot{\gamma}) \\
	& =-u_k u_j X_i g([X_k,X_\ell],X_j)  - 3 g( \dot{\gamma},\phi_\tau X_i)   g(\phi_\tau X_\ell,\dot{\gamma}) \\
	& =-u_k u_j X_i g([X_k,X_\ell],X_j)  - 3 g( X_i,\Pi_\phi X_\ell) , \qedhere
	\end{align*}	
	where we used Koszul formula, Lemma~\ref{fermi_frame} and the properties of $3$-Sasakian manifolds.
\end{proof}
In the next two lemmas, for reference, we provide many identities that will be used throughout this section. They follow from routine computations, that we omit.
\begin{lemma} We have the following identities (along the extremal):\label{l:matrixId}
\begin{align}
AA^* & = \mathbbold{1},& A^*A &= \Pi_{\phi},& \dot{A}\dot{A}^* & = 4 \|v\|^2\mathbbold{1},\\ 
A\dot{A}^* &= -\dot{A}A^*  =  2 V,& \ddot{A} &= -4 \|v\|^2A,& AC& = \tfrac{1}{2}\dot{A} - 2 v \dot\gamma^*,\\
\dot{A}C& = \left(2\|v\|^2 + 4 V^2\right)A,& ACA^*& = -V, & C^2& = -\|v\|^2\mathbbold{1},\\
AD& = 2 V, & A\dot{D}& =  4 V^2, &\dot{A}D& = -4V^2,\\
\dot{A} & = -2 VA + 2v\dot{\gamma}^*, & vv^* & = V^2 + \|v\|^2\mathbbold{1}, & V^3 &= - \|v\|^2 V, \\
B\dot{\gamma} &  = 0, & A\dot{\gamma} & = 0, & A\ddot{\gamma} & = -2v, \\
\dot{A}\dot{\gamma} & = 2v, & \dot{A} \ddot{\gamma} & = 0, & 2 C\dot{\gamma} & = \ddot{\gamma},\\
\ddot{\gamma} & = -2A^*v,& \dot{\gamma}^* D & = 0, & Vv & = 0,
\end{align}
where here $\dot\gamma$, $\ddot\gamma$ and $\Pi_\phi$ are the column vectors and the matrix that represent, respectively, the horizontal vectors $\dot\gamma$, $\ddot\gamma=\nabla_{\dot\gamma}\dot\gamma$ and the orthogonal projection $\Pi_\phi$ in the frame $\{X_i\}$.
\end{lemma}
\begin{lemma}\label{l:dots} Along the extremal, we have
	\begin{align}
	\dot\partial_v & = 2 A \partial_u, \\
	\ddot\partial_v & = 2 \dot{A} \partial_u - 2 A\vec{u}, \\
	\dddot{\partial_v} & = -8V\partial_v -2A(4\|v\|^2 +B)\partial_u +4\vec{v} +4(3VA -v\dot\gamma^*)\vec{u}, \\
	\ddddot{\partial_v} & = 
	48V^2 \partial_v -2\left[8VA(1-\|v\|^2-B) + A\dot{B} + 8\|v\|^2v \dot\gamma^* \right] \partial_u \\
	&  \quad - 24 V \vec{v} + 2\left(4\|v\|^2A +AB - 24 V^2 A\right)\vec{u} .
	\end{align}
Moreover all the non-zero brackets between $\partial_{u_i},\partial_{v_\alpha},\vec{u}_i,\vec{v}_\alpha$ are
	\begin{align}
	\sigma(\partial_u,\vec{u}) = \mathbbold{1}, \qquad \sigma(\partial_v,\vec{v}) = \mathbbold{1}, \qquad \sigma(\vec{u},\vec{u})  = -2 C, \qquad \sigma(\vec{v},\vec{v}) = 2 V .
	\end{align}
As a consequence we have
\begin{align}
\sigma(\partial_v,\partial_v) & = \sigma(\dot\partial_v,\partial_v) =0,& \sigma(\dot\partial_v,\dot\partial_v) & = \sigma(\ddot\partial_v,\partial_v) = 0,\\ 
\sigma(\ddot\partial_v,\dot\partial_v) &= 4 \mathbbold{1},& \sigma(\ddot\partial_v,\ddot\partial_v) & =  24 V,  \\
\sigma(\dddot\partial_v,\partial_v) & = - 4 \mathbbold{1},& \sigma(\dddot\partial_v,\dot\partial_v) & = -24 V,
\end{align}
\begin{align}
G:=\sigma(\dddot{\partial}_v,\ddot{\partial}_v) & = 4\left(ABA^* + 4\|v\|^2 \mathbbold{1} -24V^2\right), \\
P:=\sigma(\dddot{\partial}_v,\dddot{\partial}_v) & = 4\left( 6(VABA^* + ABA^*V) -8V + 120\|v\|^2V \right),\\
\sigma(\ddddot\partial_v,\partial_v) & = 24 V,\\
\sigma(\ddddot\partial_v,\dot\partial_v) & = - G, \\
\sigma(\ddddot\partial_v,\ddot\partial_v) & = \dot{G} - P,\\
S:=\sigma(\ddddot{\partial}_v,\dddot{\partial}_v) & = 4\left( 16\|v\|^4 + 96V^2 -480\|v\|^2V^2 - 24V^2ABA^* - 12ABA^*V^2\right.\\
& \quad \left. - 48VABA^*V  + 8\|v\|^2 ABA^* +  AB^2A^* + 6A\dot{B}A^*V\right).
\end{align}
\end{lemma}

\subsection{Canonical frame}
Following the general construction developed in \cite{lizel}, we recover the elements of the canonical frame in the following order:
\begin{equation}
E_a \rightarrow E_b \rightarrow F_b  \rightarrow E_c  \rightarrow F_c \rightarrow R_{bb}, R_{ba} \rightarrow  R_{cc}, R_{bc} \rightarrow F_a \rightarrow R_{aa},R_{ac}.
\end{equation}
The triplet $E_a$ is uniquely determined by the following conditions:
\begin{itemize}
	\item[(i)]$\pi_*E_a=0$,
	\item[(ii)]$\pi_*\dot{E}_a=0$,
	\item[(iii)] $\sigma(\ddot{E}_a,\dot{E}_a)=\mathbbold{1}$,
	\item[(iv)] $\sigma(\ddot{E}_a,\ddot{E}_a) =\mathbbold{0}$.
\end{itemize}
Conditions (i) and (ii) imply that $E^a = M \partial_v$ for $M \in \mathrm{GL}(3)$. Condition (iii) implies that $M = \tfrac{1}{2} O$ with $O \in \mathrm{O}(3)$. Finally, (iv) implies that $O$ satisfies
	\begin{equation}
	\dot{O} = \frac{1}{16} O \sigma(\ddot{\partial}_{v_\alpha},\ddot{\partial}_{v_\beta}) = \frac{3}{2}OV.
	\end{equation}
Its solution is unique up to an orthogonal transformation (the initial condition, that we set $O(0) = \mathbbold{1}$). Let us call $\V := \tfrac{3}{2} V$. Then $O(t) = e^{t \V}$ and, using the structural equations
\begin{align}
E_a & = \tfrac{1}{2} e^{t \V} \partial_{v},\\
E_b & = \dot{E}_a = \tfrac{1}{2} e^{t\V}(\V\partial_v + \dot\partial_v) , \\
F_b & = - \dot{E}_b = - \tfrac{1}{2} e^{t\V}(\V^2 \partial_v + 2 \V \dot\partial_v + \ddot{\partial}_v).  \label{eq:fb}
\end{align}
Thus we can also compute
\begin{align}
\dot{F}_b  & = - \tfrac{1}{2} e^{t\V}(\V^3 \partial_v + 3 \V^2 \dot\partial_v + 3 \V \ddot{\partial}_v + \dddot{\partial}_v), \\
\ddot{F}_b  & = - \tfrac{1}{2} e^{t\V}(\V^4 \partial_v + 4 \V^3 \dot\partial_v + 6 \V^2 \ddot\partial_v + 4 \V \dddot{\partial}_v + \ddddot{\partial_v}).
\end{align}
The next step is to compute $E_c$. It is determined by the following conditions:
	\begin{itemize}
		\item[(i)] $\pi_*E_c=0$,
		\item[(ii)] $\sigma(E_c,F_c)=\mathbbold{1}$ and $\sigma(E_c,F_b)=\sigma(E_c,F_a)=\mathbbold{0}$,
		\item[(iii)] $\pi_*\ddot{E}_c=0$.
	\end{itemize}
For (i) we can write $E_c=U\partial_u +W\partial_v$, where $U$ is a $(4d-3)\times 4d$ matrix and $W$ is a $(4d-3)\times 3$ matrix. Notice that to compute $\sigma(E_c,F_a)$ we only need to know $\pi_* F_a = -\pi_* \dot{F}_b$. Moreover, $F_c = - \dot{E}_c$. Hence, from (ii) we get
	\begin{equation}
		UU^*  =\mathbbold{1}, \qquad UA^*  =0, \qquad W  = U\dot\gamma v^*,
	\end{equation}
where $\dot\gamma$ represents, with no risk of confusion, the $4d$ dimensional column vector that represents $\dot\gamma$ in the frame $\{X_i\}$. Finally, using (iii) we get that $U$ must satisfy
	\begin{equation}\label{eq:eqforU}
	\dot{U}  =-U(\dot\gamma v^* A + C).
	\end{equation}
Observe that $U$ represents an orthogonal projection on $\distr \cap \spn\{\phi_I\dot\gamma,\phi_J\dot\gamma,\phi_K\dot\gamma\}^\perp$. Then
\begin{equation*}
	U^*U = \mathbbold{1} - \Pi_\phi = \mathbbold{1} - A^*A.
\end{equation*}
As a consequence, we have
	\begin{align}
	E_c & =U(\partial_u+\dot\gamma v^* \partial_v),\\
	F_c & =-\dot{E}_c = U[(C-\dot\gamma v^* A) \partial_u + \vec{u}], \label{eq:fc}\\
	\dot{F}_c & = U[B+\|v\|^2(1-\dot\gamma\dot\gamma^*)]\partial_u.
	\end{align}

\subsection{Sub-Riemannian curvatures}
Using the structural equations, we obtain the curvatures. We omit some very long algebraic computations, that follow using the expressions of the canonical frame obtained above.
\begin{align}
R_{bb}  & =  \sigma(\dot{F}_b,F_b) = e^{t\V}[ABA^* + 4\|v\|^2 \mathbbold{1} -\tfrac{3}{2}V^2] e^{-t\V}, \\
R_{cc} & =  \sigma(\dot{F}_c,F_c)=U[B+\|v\|^2(1-\dot\gamma\dot\gamma^*)]U^*,\\
R_{bc} & = \sigma(\dot{F}_b,F_c) = e^{t\V} A B U^*.
\end{align}
Moreover, using the structural equations and the condition $R_{ab} = -(R_{ab})^*$, we get
\begin{equation}
R_{ab}  = \tfrac{1}{2} \sigma(\dot{F}_b,\dot{F}_b) = e^{t\V}\left[\tfrac{3}{4} (VABA^*+ABA^*V)+ \tfrac{3}{2} \|v\|^2 V -4V \right]e^{-t\V}.
\end{equation}
Observe that $R_{ab}$ is correctly skew-symmetric. The last element of the frame is
\begin{equation}
F_a  = -\dot{F}_b + R_{bb}E_b + R_{ba}E_a + R_{bc}E_c \label{eq:fa}.
\end{equation}
We check that $\sigma(E_a,F_a)= \mathbbold{1}$ and $\sigma(F_a,E_b) = \sigma(F_a,F_b) = \sigma(F_a,F_a) = \sigma(F_a,F_c)= \mathbbold{0}$. Then
\begin{align}
	R_{ac} = (R_{ca})^* = \sigma(\dot{F}_c,F_a)^* =e^{t\V} \V A B U^* .
\end{align}
And finally
\begin{align}
R_{aa}  = \sigma(\dot{F}_a,F_a) & = e^{t\V}\left[\frac{3}{4} (A\dot{B}A^*V+V^*A\dot{B}A^*)+\frac{3}{8} (ABA^*V^2+ 
V^2ABA^*)  \right.\\
& \qquad \left. +3 VABA^*V^* +\left(12+\frac{45}{16} \|v\|^2 \right) V^2\right]e^{-t\V}.
\end{align}

\begin{proposition}[Canonical splitting for $3$-Sasakian structures]
	The canonical splitting along $\gamma(t)$ is $T_{\gamma(t)} M = S^a_{\gamma(t)} \oplus S^b_{\gamma(t)} \oplus S^c_{\gamma(t)}$, where
	\begin{align}
		S^a_{\gamma(t)} & = \spn\{2\xi_\alpha-2 v_\alpha\dot\gamma+\tfrac{3}{2}Z_\alpha \}_{\alpha =I,J,K},\\
		S^b_{\gamma(t)} & =  \spn\{\phi_I \dot\gamma,\phi_J\dot\gamma,\phi_K\dot\gamma\}, \\
		S^c_{\gamma(t)} & = \spn\{\phi_I \dot\gamma,\phi_J\dot\gamma,\phi_K\dot\gamma\}^\perp \cap \distr_{\gamma(t)},
	\end{align}
	where $Z_\alpha = -\sum_{\beta} v_{\alpha\beta} \phi_\beta\dot\gamma \in \distr_{\gamma(t)}$ for $\alpha = I,J,K$ and everything is computed at $\gamma(t)$.
\end{proposition}
\begin{rmk}
By Lemma~\ref{l:secder}, $Z_\alpha = \tfrac{1}{2} \phi_\alpha \nabla_{\dot\gamma}\dot\gamma + v_\alpha \dot{\gamma}$, where $\nabla$ is the Levi-Civita connection of the $3$-Sasakian structure. More explicitly
\begin{equation*}
Z_I := (v_J \phi_K  - v_K \phi_J) \dot\gamma, \qquad Z_J := (v_K \phi_I  - v_I \phi_K) \dot\gamma, \qquad Z_K := (v_I \phi_J  - v_J \phi_I) \dot\gamma.
\end{equation*}
\end{rmk}
\begin{proof}
	We project $F_a, F_b$ and $F_c$ on $T_{\gamma(t)}M$. From~\eqref{eq:fa}, \eqref{eq:fb} and \eqref{eq:fc} we get
	\begin{align}
		f_a & = \pi_*F_a = e^{t\V}(\tfrac{3}{2}VAX - 2v\dot\gamma + 2\xi), \\
	f_b & = \pi_*F_b = e^{t\V}AX, \\
	f_c & = \pi_*F_c = UX,
	\end{align}
where we recall that $X$ and $\xi$ are the tuples $\{X_i\}$ and $\{\xi_\alpha\}$ respectively. Thus, 
	\begin{align}
		\spn\{f_a\} &= \spn\{2\xi - 2 v\dot\gamma + \tfrac{3}{2}VAX \} = \spn\{2\xi_\alpha-2 v_\alpha\dot\gamma+\tfrac{3}{2} Z_\alpha\}_{\alpha =I,J,K},\\
		\spn\{f_b\} & = \spn\{AX\} = \spn\{\phi_I \dot\gamma,\phi_J\dot\gamma,\phi_K\dot\gamma\},\\
		\spn\{f_c\} & = \spn\{UX\} = \spn\{\phi_I \dot\gamma,\phi_J\dot\gamma,\phi_K\dot\gamma\}^\perp \cap \distr_{\gamma(t)}. 
	\end{align}
	Here we used the definition of $V,A$ and the fact that $U$ is a projection on the subspace of horizontal directions orthogonal to $\spn\{\phi_I \dot\gamma,\phi_J\dot\gamma,\phi_K\dot\gamma\}$.
\end{proof}

Furthermore, we summarize below the expressions for the curvature.

\begin{proposition}
	Let $M$ be a $3$-Sasakian manifold of dimension $4d+3$. In terms of the base $\{f_a(t),f_b(t),f_c(t)\}$ along a geodesic $\gamma(t)$, the canonical sub-Riemannian curvature operators $\mathfrak{R}_{\lambda(t)}^{\mu\nu} : S^\mu_{\gamma(t)} \to S^\nu_{\gamma(t)}$, for $\mu,\nu= a,b,c$, are represented by the matrices
	\begin{align}
		R_{aa}(t) & = e^{\tfrac{3}{2}tV}\left[\tfrac{3}{4} (A\dot{B}A^*V+V^*A\dot{B}A^*)+\tfrac{3}{8} (ABA^*V^2+ 
		V^2ABA^*)\right.\\
		& \qquad \left. +3 VABA^*V^*+\left(12+\tfrac{45}{16} \|v\|^2 \right) V^2\right]e^{-\tfrac{3}{2}tV},\\
		R_{ab}(t) & = e^{\tfrac{3}{2}tV}\left[\tfrac{3}{4} (VABA^*+ABA^*V) + \tfrac{3}{2}\|v\|^2 V - 4V \right]e^{-\tfrac{3}{2}tV},\\
		R_{ac}(t) & = \tfrac{3}{2}e^{\tfrac{3}{2}tV}VABU^*,\\
		R_{bb}(t) & = e^{\tfrac{3}{2}tV}[ABA^* + 4\|v\|^2 \mathbbold{1} -\tfrac{3}{2}V^2] e^{-\tfrac{3}{2}tV}, \\
		R_{bc}(t) & = e^{\tfrac{3}{2}tV}ABU^*,\\
		R_{cc}(t) & = U[B+\|v\|^2(\mathbbold{1}-\dot\gamma\dot\gamma^*)]U^*.
	\end{align}
\end{proposition}

\subsection{Proof of Theorem~\ref{t:ricci-3-sas-intro}}
We only have to compute the traces of $R_{aa}$, $R_{bb}$ and $R_{cc}$ above.  
\begin{align}
	\Riccan^a(\lambda(t)) & = \trace(R_{aa}(t)) = \tfrac{9}{4}\trace(VABA^*V^*) - (12+\tfrac{45}{16}\|v\|^2)\trace(VV^*)\\
	& = \tfrac{9}{4} \sum_{\alpha}R^\nabla(\dot\gamma,v_{\alpha\beta}\phi_{\beta}\dot\gamma,v_{\alpha\beta}\phi_{\beta}\dot\gamma,\dot\gamma) + \tfrac{27}{4}\trace(VV^*) - (12+\tfrac{45}{16}\|v\|^2)\trace(VV^*)\\
	& = \tfrac{9}{4} \sum_{\alpha}R^\nabla(\dot\gamma,Z_{\alpha},Z_{\alpha},\dot\gamma) -\tfrac{21}{2}\|v\|^2-\tfrac{45}{8}\|v\|^4, 	
\end{align}
where we used that $\trace(VV^*)=2\|v\|^2$ and we set $Z_\alpha= - v_{\alpha\beta}\phi_{\beta}\dot\gamma \in\distr_{\gamma(t)}$, for $\alpha=I,J,K$.
\begin{align}
	\Riccan^b(\lambda(t)) &= \trace(R_{bb}(t)) = \trace(ABA^*) + 12\|v\|^2 + 3\|v\|^2\\
	& = \sum_{\alpha} R^\nabla(\dot\gamma,\phi_\alpha \dot\gamma,\phi_\alpha \dot\gamma,\dot\gamma) + 9 + 15\|v\|^2 = 3(4+5\|v\|^2),	
\end{align}
where we used Proposition~\ref{p:sumholo}. Finally,
\begin{align}
	\Riccan^c(\lambda(t)) &= \trace(R_{cc}(t))=\trace((B + \|v\|^2(\mathbbold{1}-\dot\gamma\dot\gamma^*))U^*U)\\
	& = \trace((B + \|v\|^2(\mathbbold{1}-\dot\gamma\dot\gamma^*))(\mathbbold{1}-A^*A))\\
	& = \trace(B) - \trace(ABA^*) + \|v\|^2\trace(\mathbbold{1}-\dot\gamma\dot\gamma^*) - \|v\|^2\trace(A^*A)\\
	& = \sum_{i}R^\nabla(\dot\gamma,X_i,X_i,\dot\gamma) + 3\sum_{i}g(X_i,\Pi_\phi X_i)- \sum_{\alpha} R^\nabla(\dot\gamma,\phi_\alpha\dot\gamma,\phi_\alpha\dot\gamma,\dot\gamma)\\
	&\quad -9 + \|v\|^2(4d-1) - 3\|v\|^2\\
	& = \Ric^\nabla(\dot\gamma)  - \sum_{\alpha}\Sec(\dot\gamma,\xi_\alpha)- \sum_{\alpha}\Sec(\dot\gamma,\phi_\alpha\dot\gamma) + (4d-4)\|v\|^2 \\
	& = 4d+2 - 3-3+ (4d-4)\|v\|^2 = (4d-4)(1+\|v\|^2),
\end{align}
where $\Ric^\nabla(\dot\gamma)=4d+2$ by (i) of Theorem~\ref{t:BGM} and we used Propositions~\ref{p:sumholo}-\ref{p:transversecurv}. $\qed$

\section*{Acknowledgments}
This research has  been supported by the European Research Council, ERC StG 2009 ``GeCoMethods'', contract n. 239748. The first author was supported by the Grant ANR-15-CE40-0018 of the ANR, by the iCODE institute (research project of the Idex Paris-Saclay), by the SMAI project ``BOUM''. This work was supported by a public grant as part of the Investissement d'avenir project, reference ANR-11-LABX-0056-LMH, LabEx LMH, in a joint call with Programme Gaspard Monge en Optimisation et Recherche Opérationnelle.

\bibliographystyle{abbrv}
\bibliography{biblio-fat-comparison}

\def\cprime{$'$}
\begin{thebibliography}{10}

\bibitem{abou2003matrix}
H.~Abou-Kandil, G.~Freiling, V.~Ionescu, and G.~Jank.
\newblock {\em Matrix Riccati Equations: In Control and Systems Theory}.
\newblock Systems \& Control: Foundations \& Applications. Springer Verlag NY,
  2003.

\bibitem{agrachevsmooth}
A.~Agrachev.
\newblock Any sub-{R}iemannian metric has points of smoothness.
\newblock {\em Dokl. Akad. Nauk}, 424(3):295--298, 2009.

\bibitem{AAA-openproblems}
A.~Agrachev.
\newblock Some open problems.
\newblock In {\em Geometric control theory and sub-{R}iemannian geometry},
  volume~5 of {\em Springer INdAM Ser.}, pages 1--13. Springer, Cham, 2014.

\bibitem{ABP-distortion}
A.~{Agrachev}, D.~{Barilari}, and E.~{Paoli}.
\newblock {Volume geodesic distortion and Ricci curvature for Hamiltonian
  dynamics}.
\newblock {\em ArXiv e-prints}, Feb. 2016.

\bibitem{ABR-curvature}
A.~Agrachev, D.~Barilari, and L.~Rizzi.
\newblock Curvature: a variational approach.
\newblock {\em Memoirs of the AMS (in press)}, 2015.

\bibitem{ABR-curvaturecontact}
A.~Agrachev, D.~Barilari, and L.~Rizzi.
\newblock Sub-{R}iemannian {C}urvature in {C}ontact {G}eometry.
\newblock {\em J. Geom. Anal.}, 27(1):366--408, 2017.

\bibitem{ABGR-unimodular}
A.~Agrachev, U.~Boscain, J.-P. Gauthier, and F.~Rossi.
\newblock The intrinsic hypoelliptic {L}aplacian and its heat kernel on
  unimodular {L}ie groups.
\newblock {\em J. Funct. Anal.}, 256(8):2621--2655, 2009.

\bibitem{agrafeedback}
A.~Agrachev and R.~V. Gamkrelidze.
\newblock Feedback-invariant optimal control theory and differential geometry.
  {I}. {R}egular extremals.
\newblock {\em J. Dynam. Control Systems}, 3(3):343--389, 1997.

\bibitem{AL-3D-MCP}
A.~Agrachev and P.~W.~Y. Lee.
\newblock Generalized {R}icci curvature bounds for three dimensional contact
  subriemannian manifolds.
\newblock {\em Math. Ann.}, 360(1-2):209--253, 2014.

\bibitem{AL-3D-BisLap}
A.~Agrachev and P.~W.~Y. Lee.
\newblock Bishop and {L}aplacian comparison theorems on three-dimensional
  contact sub-{R}iemannian manifolds with symmetry.
\newblock {\em J. Geom. Anal.}, 25(1):512--535, 2015.

\bibitem{ARS-LQ}
A.~Agrachev, L.~Rizzi, and P.~Silveira.
\newblock On {C}onjugate {T}imes of {LQ} {O}ptimal {C}ontrol {P}roblems.
\newblock {\em J. Dyn. Control Syst.}, 21(4):625--641, 2015.

\bibitem{Agrachevbook}
A.~Agrachev and Y.~Sachkov.
\newblock {\em Control theory from the geometric viewpoint}, volume~87 of {\em
  Encyclopaedia of Mathematical Sciences}.
\newblock Springer-Verlag, Berlin, 2004.
\newblock Control Theory and Optimization, II.

\bibitem{agzel1}
A.~Agrachev and I.~Zelenko.
\newblock Geometry of {J}acobi curves. {I}.
\newblock {\em J. Dynam. Control Systems}, 8(1):93--140, 2002.

\bibitem{agzel2}
A.~Agrachev and I.~Zelenko.
\newblock Geometry of {J}acobi curves. {II}.
\newblock {\em J. Dynam. Control Systems}, 8(2):167--215, 2002.

\bibitem{nostrolibro}
A.~A. Agrachev, D.~Barilari, and U.~Boscain.
\newblock Introduction to {R}iemannian and sub-{R}iemannian geometry ({L}ecture
  {N}otes).
  http://webusers.imj-prg.fr/{\textasciitilde}davide.barilari/notes.php.
\newblock 2016, June 12th.

\bibitem{bbcn}
D.~{Barilari}, U.~{Boscain}, G.~{Charlot}, and R.~W. {Neel}.
\newblock {On the heat diffusion for generic Riemannian and sub-Riemannian
  structures}.
\newblock {\em Int. Math. Res. Notices (in press)}, Oct. 2016.

\bibitem{srneel}
D.~Barilari, U.~Boscain, and R.~W. Neel.
\newblock Small-time heat kernel asymptotics at the sub-{R}iemannian cut locus.
\newblock {\em J. Differential Geom.}, 92(3):373--416, 2012.

\bibitem{BI-QC}
D.~{Barilari} and S.~{Ivanov}.
\newblock {A sub-Riemannian Bonnet-Myers theorem for quaternionic contact
  structures}.
\newblock {\em ArXiv e-prints}, Mar. 2017.

\bibitem{BR-Popp}
D.~Barilari and L.~Rizzi.
\newblock A formula for {P}opp's volume in sub-{R}iemannian geometry.
\newblock {\em Anal. Geom. Metr. Spaces}, 1:42--57, 2013.

\bibitem{BR-connection}
D.~{Barilari} and L.~{Rizzi}.
\newblock {On Jacobi fields and canonical connection in sub-Riemannian
  geometry}.
\newblock {\em Archivum Mathematicum (in press)}, June 2015.

\bibitem{BR-comparison}
D.~Barilari and L.~Rizzi.
\newblock Comparison theorems for conjugate points in sub-{R}iemannian
  geometry.
\newblock {\em ESAIM Control Optim. Calc. Var.}, 22(2):439--472, 2016.

\bibitem{garobonneinequalities}
F.~Baudoin, M.~Bonnefont, and N.~Garofalo.
\newblock A sub-{R}iemannian curvature-dimension inequality, volume doubling
  property and the {P}oincar\'e inequality.
\newblock {\em Math. Ann.}, 358(3-4):833--860, 2014.

\bibitem{baudoinmunivegarofalo}
F.~Baudoin, M.~Bonnefont, N.~Garofalo, and I.~H. Munive.
\newblock Volume and distance comparison theorems for sub-{R}iemannian
  manifolds.
\newblock {\em J. Funct. Anal.}, 267(7):2005--2027, 2014.

\bibitem{garofalob}
F.~{Baudoin} and N.~{Garofalo}.
\newblock {Curvature-dimension inequalities and Ricci lower bounds for
  sub-Riemannian manifolds with transverse symmetries}.
\newblock {\em To appear in the Journal of the European Mathematical Society},
  2014.

\bibitem{BK-Lich-Obata}
F.~Baudoin and B.~Kim.
\newblock The {L}ichnerowicz-{O}bata theorem on sub-{R}iemannian manifolds with
  transverse symmetries.
\newblock {\em J. Geom. Anal.}, 26(1):156--170, 2016.

\bibitem{BKW-Weitzen}
F.~{Baudoin}, B.~{Kim}, and J.~{Wang}.
\newblock {Transverse {Weitzenb\"ock} formulas and curvature dimension
  inequalities on Riemannian foliations with totally geodesic leaves}.
\newblock {\em To appear in Communications in Analysis and Geometry}, Aug.
  2014.

\bibitem{BW-CR}
F.~Baudoin and J.~Wang.
\newblock The subelliptic heat kernel on the {CR} sphere.
\newblock {\em Math. Z.}, 275(1-2):135--150, 2013.

\bibitem{baudoincontact}
F.~Baudoin and J.~Wang.
\newblock Curvature-dimension inequalities and subelliptic heat kernel gradient
  bounds on contact manifolds.
\newblock {\em Potential Analysis}, 40:163--193, 2014.

\bibitem{BW-QHF}
F.~Baudoin and J.~Wang.
\newblock The subelliptic heat kernels of the quaternionic {H}opf fibration.
\newblock {\em Potential Anal.}, 41(3):959--982, 2014.

\bibitem{blair}
D.~E. Blair.
\newblock {\em Riemannian geometry of contact and symplectic manifolds}, volume
  203 of {\em Progress in Mathematics}.
\newblock Birkh\"auser Boston, Inc., Boston, MA, second edition, 2010.

\bibitem{BGM-3-Sasakian}
C.~P. Boyer, K.~Galicki, and B.~M. Mann.
\newblock The geometry and topology of {$3$}-{S}asakian manifolds.
\newblock {\em J. Reine Angew. Math.}, 455:183--220, 1994.

\bibitem{Coronbook}
J.-M. Coron.
\newblock {\em Control and nonlinearity}, volume 136 of {\em Mathematical
  Surveys and Monographs}.
\newblock American Mathematical Society, Providence, RI, 2007.

\bibitem{esch}
J.-H. Eschenburg and E.~Heintze.
\newblock Comparison theory for {R}iccati equations.
\newblock {\em Manuscripta Math.}, 68(2):209--214, 1990.

\bibitem{GHL-RiemannianGeom}
S.~Gallot, D.~Hulin, and J.~Lafontaine.
\newblock {\em Riemannian geometry}.
\newblock Universitext. Springer-Verlag, Berlin, third edition, 2004.

\bibitem{HS-remarks}
I.~Hasegawa and M.~Seino.
\newblock Some remarks on {S}asakian geometry---applications of {M}yers'
  theorem and the canonical affine connection.
\newblock {\em J. Hokkaido Univ. Ed. Sect. II A}, 32(1):1--7, 1981/82.

\bibitem{Hladky-qc}
R.~K. Hladky.
\newblock The topology of quaternionic contact manifolds.
\newblock {\em Ann. Global Anal. Geom.}, 47(1):99--115, 2015.

\bibitem{HandbookNORMALFRAMES}
B.~Z. Iliev.
\newblock {\em Handbook of normal frames and coordinates}, volume~42 of {\em
  Progress in Mathematical Physics}.
\newblock Birkh\"auser Verlag, Basel, 2006.

\bibitem{IMV-QCYamabe}
S.~Ivanov, I.~Minchev, and D.~Vassilev.
\newblock Quaternionic contact {E}instein structures and the quaternionic
  contact {Y}amabe problem.
\newblock {\em Mem. Amer. Math. Soc.}, 231(1086):vi+82, 2014.

\bibitem{IMV-QCEinsten}
S.~Ivanov, I.~Minchev, and D.~Vassilev.
\newblock Quaternionic contact {E}instein manifolds.
\newblock {\em Math. Res. Lett.}, 23(5):1405--1432, 2016.

\bibitem{iv15}
S.~Ivanov and D.~Vassilev.
\newblock The {L}ichnerowicz and {O}bata first eigenvalue theorems and the
  {O}bata uniqueness result in the {Y}amabe problem on {CR} and quaternionic
  contact manifolds.
\newblock {\em Nonlinear Analysis: Theory, Methods {\&} Applications}, 126:262
  -- 323, 2015.

\bibitem{Jurdjevicbook}
V.~Jurdjevic.
\newblock {\em Geometric control theory}, volume~52 of {\em Cambridge Studies
  in Advanced Mathematics}.
\newblock Cambridge University Press, Cambridge, 1997.

\bibitem{Kashiwada-contact3}
T.~Kashiwada.
\newblock On a contact 3-structure.
\newblock {\em Math. Z.}, 238(4):829--832, 2001.

\bibitem{LL-Sasakian-BisLap}
P.~W.~Y. Lee and C.~Li.
\newblock {Bishop and Laplacian comparison theorems on Sasakian manifolds}.
\newblock {\em ArXiv e-prints}, Oct. 2013.

\bibitem{LLZ-Sasakian-MCP}
P.~W.~Y. Lee, C.~Li, and I.~Zelenko.
\newblock Ricci curvature type lower bounds for sub-{R}iemannian structures on
  {S}asakian manifolds.
\newblock {\em Discrete Contin. Dyn. Syst.}, 36(1):303--321, 2016.

\bibitem{lizel2}
C.~Li and I.~Zelenko.
\newblock Jacobi equations and comparison theorems for corank 1
  sub-{R}iemannian structures with symmetries.
\newblock {\em J. Geom. Phys.}, 61(4):781--807, 2011.

\bibitem{Mis-FermiNormal}
F.~K. Manasse and C.~W. Misner.
\newblock Fermi normal coordinates and some basic concepts in differential
  geometry.
\newblock {\em J. Mathematical Phys.}, 4:735--745, 1963.

\bibitem{montgomerybook}
R.~Montgomery.
\newblock {\em A tour of subriemannian geometries, their geodesics and
  applications}, volume~91 of {\em Mathematical Surveys and Monographs}.
\newblock American Mathematical Society, Providence, RI, 2002.

\bibitem{Myers}
S.~B. Myers.
\newblock Riemannian manifolds with positive mean curvature.
\newblock {\em Duke Math. J.}, 8:401--404, 1941.

\bibitem{N-bound}
Y.~Nitta.
\newblock A diameter bound for {S}asaki manifolds.
\newblock {\em Ann. Sc. Norm. Super. Pisa Cl. Sci. (5)}, 13(1):207--224, 2014.

\bibitem{rifford2014sub}
L.~Rifford.
\newblock {\em Sub-Riemannian Geometry and Optimal Transport}.
\newblock SpringerBriefs in Mathematics. Springer, 2014.

\bibitem{Roydencomp}
H.~L. Royden.
\newblock Comparison theorems for the matrix {R}iccati equation.
\newblock {\em Comm. Pure Appl. Math.}, 41(5):739--746, 1988.

\bibitem{strichartz}
R.~S. Strichartz.
\newblock Sub-{R}iemannian geometry.
\newblock {\em J. Differential Geom.}, 24(2):221--263, 1986.

\bibitem{strichartzerrata}
R.~S. Strichartz.
\newblock Corrections to: ``{S}ub-{R}iemannian geometry'' [{J}. {D}ifferential
  {G}eom.\ {\bf 24} (1986), no.\ 2, 221--263; {MR}0862049 (88b:53055)].
\newblock {\em J. Differential Geom.}, 30(2):595--596, 1989.

\bibitem{Tanno}
S.~Tanno.
\newblock Killing vectors on contact {R}iemannian manifolds and fiberings
  related to the {H}opf fibrations.
\newblock {\em T\^ohoku Math. J. (2)}, 23:313--333, 1971.

\bibitem{villanibook}
C.~Villani.
\newblock {\em Optimal transport}, volume 338 of {\em Grundlehren der
  Mathematischen Wissenschaften [Fundamental Principles of Mathematical
  Sciences]}.
\newblock Springer-Verlag, Berlin, 2009.
\newblock Old and new.

\bibitem{lizel}
I.~Zelenko and C.~Li.
\newblock Differential geometry of curves in {L}agrange {G}rassmannians with
  given {Y}oung diagram.
\newblock {\em Differential Geom. Appl.}, 27(6):723--742, 2009.

\end{thebibliography}

\end{document}